\documentclass[11pt,a4paper]{amsart}

\usepackage{tikz}
\usetikzlibrary{arrows.meta,decorations.pathmorphing,decorations.markings,backgrounds,positioning,fit,shapes}

\usepackage[all,cmtip]{xy}
\usepackage[margin=1in]{geometry}  %
\usepackage{graphicx}              %
\usepackage{amsmath,euscript,comment}               %
\usepackage{amsfonts}              %
\usepackage{amssymb,amscd,epsf,amsthm, epsfig,bm,leftindex,mathabx} 

\usepackage{hyperref}
\usepackage{mathbbol}
\usepackage{enumerate}              %
\usepackage[toc,page]{appendix}
\usepackage{amsmath,calligra,mathrsfs,xcolor,mathtools}
\usepackage{float}
\usepackage[symbol]{footmisc}
\usepackage{xcolor}

\usepackage[capitalize,nameinlink]{cleveref}

\usepackage[scr=boondox,  %
            cal=esstix]   %
           {mathalpha}

\usetikzlibrary{shapes.misc}
\tikzset{>={Stealth[scale=1.2]}}
\tikzset{->-/.style={decoration={
			markings,
			mark=at position #1 with {\arrow{>}}},postaction={decorate}}}
\tikzset{-w-/.style={decoration={
			markings,
			mark=at position #1 with {\arrow{Stealth[fill=white,scale=1.4]}}},postaction={decorate}}}
\tikzset{->-/.default=0.65}
\tikzset{-w-/.default=0.65}
\tikzstyle{bullet}=[circle,fill=black,inner sep=0.5mm]
\tikzstyle{wullet}=[circle,draw=black,fill=white,inner sep=0.5mm,line width=0.5pt]
\tikzstyle{cross}=[cross out, draw=black, inner sep=4.8pt, line width = 1pt]
\tikzstyle{circ}=[circle,draw=black,fill=white,inner sep=5pt,line width=1pt]
\tikzstyle{vertex}=[circle,draw=black,thick,inner sep=0.5mm]
\tikzset{darrow/.style={double distance = 4pt,>={Implies},->},
	darrowthin/.style={double equal sign distance,>={Implies},->},
	tarrow/.style={-,preaction={draw,darrow}},
	qarrow/.style={preaction={draw,darrow,shorten >=0pt},shorten >=1pt,-,double,double
		distance=0.2pt}}
\newcommand{\tikzfig}[1]{\begin{tikzpicture}[auto,baseline={([yshift=-.5ex]current bounding box.center)}]#1\end{tikzpicture}}
\tikzset{partial ellipse/.style args={#1:#2:#3}{insert path={+ (#1:#3) arc (#1:#2:#3)}}}
 \usepackage{tikz-cd} 
\newtheorem{theorem}{Theorem}[section]

\newtheorem{proposition}[theorem]{Proposition}
\newtheorem{corollary}[theorem]{Corollary}

\theoremstyle{definition}
\newtheorem{definition}[theorem]{Definition}

\theoremstyle{remark}
\newtheorem{example}[theorem]{Example}

\newtheorem{remark}[theorem]{Remark}

\numberwithin{theorem}{section}

\newcommand{\R}{\mathbb{R}}

\newcommand{\calC}{\mathcal{C}}

\newcommand{\cals}{\mathcal{s}}

\newcommand{\Hom}{\mathop{\mathrm{Hom}}\nolimits}
\newcommand{\Span}{\mathop{\mathrm{Span}}\nolimits}

\newcommand{\Imm}{\mathop{\mathrm{Im}}\nolimits}

\newcommand{\id}{\mathop{\mathrm{id}}\nolimits}

\newcommand{\Supp}{\mathop{\mathrm{Supp}}\nolimits}

\newcommand{\diam}{\mathop{\mathrm{diam}}\nolimits}

\newcommand{\coCH}{\mathop{\mathcal{coCH}}\nolimits}

\mathchardef\mh="2D

\input xy
\xyoption{all}

\usepackage{lipsum}

\usepackage{hyperref}
\usepackage[backend=biber, maxbibnames=99, bibencoding=utf8, doi=false, isbn=false, url=false, style=alphabetic]{biblatex}
\begin{document}

\title{String topology via the coHochschild complex and local intersections}
\author{Manuel Rivera and Alex Takeda}

\newcommand{\Addresses}{{%
  \bigskip
  \footnotesize

  M.~Rivera, \textsc{Purdue University, Department of Mathematics, 150 N. University St. West Lafayette, IN 47907}\par\nopagebreak
  \textit{E-mail address}: \texttt{manuelr@purdue.edu}
  
   \medskip
   A.~Takeda, \textsc{Uppsala University, Department of Mathematics, Box 256, 751 05 Uppsala, Sweden}\par\nopagebreak
  \textit{E-mail address}: \texttt{alex.takeda@math.uu.se}
 }}

\begin{abstract}
    We construct an algebraic model for the Chas--Sullivan product and the Goresky--Hingston coproduct in string topology. The construction takes as its initial input a simplicial complex equipped with a local pairing on its simplicial chains, for instance, a homology manifold with its local intersection pairing. We define the two string topology operations on the coHochschild complex of a suitable coalgebra of chains, making use of local higher homotopies that control the compatibility of the pairing with the diagonal approximation coproduct. In the case of a closed oriented smooth manifold, we prove that our algebraic operations coincide, up to chain homotopy, with their geometric counterparts. The local nature of our constructions allows for arguments based on the method of acyclic models.
\end{abstract}

\maketitle

\vspace{-0.5cm} 

\section{Introduction}

In this article, we construct an explicit and tractable algebraic model for two fundamental operations in string topology: the Chas--Sullivan product \cite{CS} and the Goresky--Hingston coproduct \cite{GH}. Our approach makes no hypotheses about the connectivity of the underlying space, allows for an arbitrary commutative ring of coefficients, and extends previous models known in the simply connected setting that are based on rational homotopy theory, Hochschild homology, and Frobenius algebras \cite{LS,Abb,NW,RW}. Furthermore, our framework elucidates the necessity of locality conditions to obtain geometrically meaningful chain-level string operations and allows us to define the Goresky-Hingston coproduct in the context of homology manifolds. 

Our construction starts with a simplicial complex equipped with a pairing satisfying an appropriate locality condition and relies on two main ingredients:
\begin{enumerate}
    \item a model for the free loop space, given by the coHochschild complex of a suitable coalgebra associated to the simplicial complex, and
    \item a coherent family of higher homotopies that extend the pairing, are also appropriately local, and control the compatibility between the pairing and the diagonal approximation coproduct.
\end{enumerate}

The model for the free loop space in (1) follows the work of the first author in \cite{rivera2023algebraic} and has the following form. Suppose $X$ is a simplicial set, e.g. a simplicial complex with a coherent ordering of the vertices of each simplex and formal degeneracies. Replace $X$ by a homotopy equivalent simplicial set $\widetilde{X}$ with the property that every edge is invertible up to homotopy, so that its homotopy category is a groupoid. This can be done functorially by gluing the nerve of the groupoid with two objects and two non-identity morphisms along every edge. This step may be bypassed if $X$ is simply connected by contracting its $1$-skeleton. Fix a commutative ring $R$ and consider the normalized simplicial $R$-chains $\mathcal{C}(\widetilde{X})$ with Alexander-Whitney diagonal approximation coproduct. With a slight modification of the simplicial boundary map, $\mathcal{C}(\widetilde{X})$ becomes a \textit{categorical coalgebra}: a curved coalgebra on a set of objects (the vertices of $X$) satisfying certain basic properties. The \textit{coHochschild complex} \footnote[1]{The coHochschild complex can be defined for any categorical coalgebra $C$ and it calculates the derived cotensor product of $C$ with itself in a suitable homotopy theory of $C$-bicomodules.} of $\mathcal{C}(\widetilde{X})$, denoted by $\coCH_*(\mathcal{C}(\widetilde{X}))$, provides a chain model for the free loop space of the geometric realization $|X|$. The underlying graded $R$-module of $\coCH_*(\mathcal{C}(\widetilde{X}))$ is generated by \textit{necklaces} in $\widetilde{X}$. A necklace is a cyclically ordered sequence of simplices in $\widetilde{X}$, called the \textit{beads} of the necklace, such that the last vertex of a simplex coincides with the first vertex of the next one and one the beads, called the \textit{marked bead}, is distinguished. Necklaces represent families of loops in $|X|$ with base points in the marked bead.  
Furthermore, the complex of constant loops inside $\coCH(\mathcal{C}(\widetilde{X}))$, which will play an important role in our construction of the string topology coproduct, can be described in terms of sufficiently ``local" necklaces. The relevant notions, constructions, and results mentioned above are discussed in \cref{sec:modelingfreeloopspace,,sec:invertingedges,,sec:constantloops}.

The notion describing the structure in (2) is inspired by the pre-Calabi-Yau formalism developed by the second author, Kontsevich, and Vlassopoulos in \cite{KTV}. This chain-level structure may be constructed for any simplicial complex equipped with a local intersection pairing as is the case of a triangulated smooth oriented compact manifold or, more generally, a homology manifold. For instance, if $M$ is a smooth oriented compact manifold of dimension $n$, Poincaré duality gives rise to an intersection pairing
\[ \cap \colon H_*(M) \otimes H_*(M) \to R\] on the homology of $M$. Present the homotopy type of $M$ by a simplicial set $X_M$ obtained from a sufficiently fine triangulation and homotopically invert edges to obtain $\widetilde{X}_M$ as above. The required structure consists of a degree $-n$ map 
\[ \alpha \colon \mathcal{C}(\widetilde{X}_M) \otimes \mathcal{C}(\widetilde{X}_M)  \to \Omega \mathcal{C}(\widetilde{X}_M)  \otimes \Omega \mathcal{C}(\widetilde{X}_M),\]
where $\Omega \mathcal{C}(\widetilde{X}_M)$ denotes the cobar construction of $\mathcal{C}(\widetilde{X}_M)$, which is closed with respect to a certain differential and lifts $\cap$ in an appropriate sense. The map $\alpha$ encodes a system of maps 
reflecting the homotopical compatibility of the chain-level intersection pairing with the coproduct of $\mathcal{C}(\widetilde{X}_M)$. We call this structure a \textit{homotopy pairing} on $\mathcal{C}(\widetilde{X}_M)$ and it can be viewed as a coalgebraic analogue of a $2$-\textit{truncated pre-Calabi Yau algebra} \cite{KTV} or, in the language of \cite{TZ}, of a $V_2$-\textit{algebra}. Furthermore, $\alpha$ must satisfy a crucial \textit{locality} condition: it is non-zero only when the inputs are sufficiently close chains and, in such a case, its outputs remain accordingly close. Similar notions and constructions are considered in \cite{TZPD}; however, their approach assumes rational coefficients and, in order to recover the loop product of string topology and its $BV$-algebra structure, one assumes the underlying space is simply connected. 

The algebraic notion of a homotopy pairing can be defined for any categorical coalgebra. The geometric locality condition can be defined for any homotopy pairing on the categorical coalgebra of chains of a simplicial set arising from a simplicial complex. After recalling the role of local pairings in topology in \cref{sec:localpairings} and introducing homotopy pairings in \cref{sec:homotopypairings}, we prove that, after taking a barycentric subdivision to make the simplicial complex fine enough \cref{def:fine}, any local pairing extends to a local homotopy pairing, see \cref{thm:alphaExistence} and \cref{prop:alphaUnique} for precise statements. 

\begin{theorem}
    Let $X$ be a simplicial set determined by a coherent ordering of the simplices of a sufficiently fine simplicial complex $K$. Any degree $-n$ local pairing \[\theta \colon C_*(K) \otimes C_*(K) \to R\] on the chains of $K$ lifts to a local homotopy pairing
    \[ \alpha \colon \mathcal{C}(\widetilde{X}) \otimes \mathcal{C}(\widetilde{X})  \to \Omega \mathcal{C}(\widetilde{X})  \otimes \Omega \mathcal{C}(\widetilde{X}),\]
    with cohomology class uniquely determined by $[\theta]\in H^n(K \times K)$. In particular, the intersection pairing of a homology manifold can be lifted to a local homotopy pairing on its categorical coalgebra of chains. 
\end{theorem}

Given a homotopy pairing $\alpha$ on the categorical coalgebra of chains $\mathcal{C}(\widetilde{X})$ on a simplicial set $X$, we explicitly construct a degree $-n$ product
\[ \mu_\alpha \colon \coCH_*(\mathcal{C}(\widetilde{X})) \otimes \coCH_*(\mathcal{C}(\widetilde{X})) \to \coCH_*(\mathcal{C}(\widetilde{X}))\]
and a degree $1-n$ coproduct
\[ \lambda_\alpha \colon \coCH_*(\mathcal{C}(\widetilde{X})) \to \coCH_*(\mathcal{C}(\widetilde{X}))  \otimes \coCH_*(\mathcal{C}(\widetilde{X}
)),\]
see \cref{def:algebraicLoopProduct,def:algloopcoproduct}. The product $\mu_{\alpha}$ uses $\alpha$ to ``intersect" the marked beads of two necklaces and then concatenating the resulting beads. The coproduct $\lambda_{\alpha}$ starts with a necklace and uses $\alpha$ to intersect the marked bead with each of the other beads, taking the sum and splitting every term into two necklaces. These two algebraic operations are models for the geometrically defined Chas-Sullivan product and Goresky-Hingston coproduct in the sense of the following statements, which summarize our main results.

\begin{theorem}
    Let $X$ be a simplicial set determined by a coherent ordering of the vertices of each simplex in a simplicial complex and $\alpha$ a homotopy pairing on the categorical coalgebra of chains $\mathcal{C}(\widetilde{X})$. 
    \begin{enumerate}
        \item \emph{(\cref{prop:product})} The product $\mu_\alpha$ is a chain map, i.e., $\mu_{\alpha}$ satisfies the Leibniz rule with respect to the coHochschild differential.
        \item  \emph{(\cref{prop:coprodctmodconstant})} If $\alpha$ is local and $X$ is sufficiently fine, the failure of the coproduct $\lambda_{\alpha}$ to be a chain map, i.e., the commutator of $\lambda_{\alpha}$ and the coHochschild differential, is homotopic to a chain of constant loops. 
        \item \emph{(\cref{cor:loopproduct,cor:loopcoproduct})} If $X=X_M$ arises from a sufficiently fine triangulation of a smooth oriented closed manifold $M$ and $\alpha$ is any local homotopy pairing lifting the intersection pairing of $M$, then the quasi-isomorphism from $\coCH_*(\mathcal{C}(\widetilde{X}
        ))$ to the singular chains on $LM$ intertwines up to homotopy the operations $\mu_{\alpha}$ and $\lambda_{\alpha}$, and the geometrically defined loop product and coproduct of \cite{HW}, respectively. 
    \end{enumerate}
\end{theorem}

The operations $\mu_\alpha$ and $\lambda_\alpha$ extend to the setting of possibly non-simply connected spaces and arbitrary coefficients previous formulas defined in the context of simply connected closed manifolds and rational coefficients \cite{Abb, naef2023string, NW}. In that more restricted setting, a strict version of $\alpha$, but not satisfying any geometric locality properties, can be provided by the Poincar\'e duality commutative differential graded algebra models of Lambrechts and Stanley \cite{LS}. The locality properties of our constructions allow us to apply arguments based on the method of acyclic models to prove the above results. Note that while the product $\mu_{\alpha}$ passes to homology for any homotopy pairing $\alpha$, in order for the coproduct $\lambda_{\alpha}$ to be meaningful, we assume $\alpha$ is local. 

One of our objectives is to isolate the essential ingredients used in string topology, motivated by the fact that the loop coproduct on homology is already sensitive to structure beyond the homotopy type of the underlying manifold \cite{Naef, NS, KP, naef2023string}. Our construction produces a loop product operation in the context of Poincaré duality complexes and a loop coproduct operation in the context of Poincar\'e duality complexes whose intersection pairing has a local representative. The existence of such a local representative is equivalent to being a homology manifold, as explained in \cite{ranicki1992algebraic, ranicki1999singularities, mccrory77, mccrory1979zeeman} and recalled in \cref{sec:localpairings}.

We conjecture that the full package of chain-level string topology, in particular the $\infty$-involutive Lie bialgebra (or $IBL_{\infty}$-algebra) structure on the $S^1$-equivariant chains of the free loop space modulo constant loops, depends on a lift of the intersection product as a coalgebraic version of a pre-Calabi–Yau (or $V_{\infty}$) algebra satisfying appropriate locality conditions. For the operations of interest in this paper, a $2$-truncated version of the structure suffices. 

\subsection*{Notation and conventions}
Throughout this paper, we fix a commutative ring $R$. We just write `complexes' to mean homologically-graded complexes of $R$-modules, also known as dg $R$-modules, with a differential of degree $-1$. Whenever we write `linear', `module' etc. we shall mean $R$-linear and $R$-module, and so on. We will write $\otimes=\otimes_R$ for the tensor product of $R$-modules. If $C$ is an $R$-coalgebra, $M$ a right $C$-comodule, and $N$ a left $C$-comodule, we will write $M \boxtimes N$ for the cotensor product of $M$ and $N$ over $C$. We follow the Koszul sign rule when applying maps in the graded setting. 

We use both simplicial sets and simplicial complexes in this paper. We will use the convention that simplices in a simplicial complex are \emph{closed}, that is, contain their faces. For ease of notation, we will not distinguish between the `abstract' $n$-simplex and the `topological' $n$-simplex, its realization, and denote both by the symbol $\triangle^n$; it will be clear which one we mean from the context.  We will use the same notation $|-|$ for the (thin) geometric realization of a simplicial set and for the polytope of a simplicial complex.

We will use different types of chains and distinguish them by notation. In particular, we will use $S_*(Y)$ for singular (simplicial) chains in a topological space $Y$ and reserve $C_*(-)$ for either the normalized simplicial chains in a simplicial set or the simplicial chains in a simplicial complex, that is, finite linear combinations of oriented simplices.

\subsection*{Acknowledgments} The first author acknowledges the excellent working conditions at Northwestern University, where part of this research was carried out and where valuable discussions with E. Getzler, F. Naef, A. Oancea, S. Saneblidze, D. Sullivan, B. Tsygan, and M. Zeinalian took place during the workshop \textit{Homotopical Algebra in Geometry, Topology, and Physics}. Both the first author and the workshop received support from the NSF grant DMS-245405. The second author also thanks H. Bursztyn, T. Ekholm, T. Kragh, J. Qiu, B. Vallette, N. Wahl, and T. Willwacher for helpful discussions, Uppsala University for the wonderful working environment, and the Knut and Alice Wallenberg foundation for the support.

\section{Modeling the free loop space}\label{sec:modelingfreeloopspace}
We describe an algebraic construction that produces a model for the chains on the free loop space directly from a suitable chain model of the underlying space.

\subsection{Categorical coalgebras}
We shall regard the normalized chains on a simplicial set as a `coalgebra with many objects' through the notion of a categorical coalgebra. Recall that a \textit{categorical coalgebra} consists of the following data:
\begin{itemize}
    \item a non-negatively graded flat module $C = \bigoplus_{i=0}^\infty C_i$,
    \item a linear counital coassociative coproduct $\Delta \colon C \to C\otimes C$ of degree $0$ (with counit map $\varepsilon \colon C \to R$) whose \emph{set of objects} (called \emph{set-like elements} in \emph{op.cit.})
    \[ \mathscr{S}_C \coloneqq \{x \in C\ |\ \Delta(x)=x\otimes x, \epsilon(x)=1\} \]
    generates $C_0$; namely the inclusion map $\mathcal{S}_C \hookrightarrow C_0$ induces an isomorphism $R[\mathscr{S}_C] \cong C_0$, 
    \item a linear map $d: C\to C$ of degree $-1$ on $C$ which is a graded coderivation for $\Delta$ and vanishes on $C_1$, and
    \item a linear map $\eta \colon C \to R$ of degree $-2$, called the \textit{curvature}, satisfying $\eta \circ d = 0$ and
    \begin{align} \label{eq:curvature} 
        d^2 = (\eta \otimes \id - \id \otimes\ \eta)\circ \Delta.
    \end{align}
\end{itemize}
Any categorical coalgebra $(C,d, \Delta, \eta)$ has a natural $C_0$-bicomodule structure (compatible with $d$) with structure maps given by
\begin{align} \label{eq:bicomoduleMaps}
    \rho_r \colon C\xrightarrow{\Delta} C\otimes C \xrightarrow{\text{id} \otimes \pi_0} C \otimes C_0 \quad \text{and} \quad \rho_l \colon C \xrightarrow{\Delta} C \otimes C\xrightarrow{\pi_0 \otimes \text{id} } C_0 \otimes C,
\end{align}
where $\pi_0 \colon C \to C_0$ is the natural projection. Categorical coalgebras form a category by defining a \textit{morphism} from $(C, d, \Delta,\eta)$ to $(C',d',\Delta', \eta')$ to consist of a pair $f=(f_0,f_1)$ where
\begin{enumerate}
    \item $f_0:(C,\Delta)\rightarrow (C',\Delta')$ is a morphism of graded coalgebras 
    \item $f_1: C \rightarrow C_0^\prime$ is a $C_0^{\prime}$-bicomodule map of degree $-1$
\end{enumerate}
satisfying
\begin{align*}
    & f_0\circ d = d' \circ f_0 + (\overline{f}_1 \otimes f_0)\circ(\Delta - \Delta^{\mathrm{op}}) \text{ and}\\
    & \eta'\circ f_0 = \eta+ \overline{f}_1 \circ d + (\overline{f}_1\otimes \overline{f}_1)\circ \Delta , 
\end{align*}
where $\overline{f}_1 = \varepsilon' \circ f_1$ and $\varepsilon'$ is the counit of $C'$. The composition of two morphisms is defined by
\[  (g_0,g_1)\circ(f_0,f_1)= (g_0 \circ f_0, g_1\circ f_0 + g_0 \circ f_1).
\]
Denote by $\mathsf{cCoalg}_R$ the category of categorical coalgebras.

\subsubsection{Normalized chains} Our main example of a categorical coalgebra is given by a suitable variation of the differential graded (dg) coalgebra of normalized chains on a simplicial set $X$. Denote by $C_*(X)=(\bigoplus_{i=0}^{\infty}C_i(X), \Delta, \delta)$ the dg coalgebra of normalized simplicial chains; namely, $C_k(X)= \Span_R(X_k)/D(X_k)$, where $D(X_k)$ denotes the submodule of $\Span_R(X_k)$ generated by degenerate $k$-simplices in $X$, \[\delta \colon C_*(X) \to C_{*-1}(X)\] is the classical simplicial boundary map defined as the alternating sum of faces, and \[\Delta \colon C_*(X) \to C_*(X) \otimes C_*(X)\] the Alexander-Whitney diagonal approximation coproduct defined on any $\sigma \in X_n$ by

\[\Delta(\sigma) = \sum_{i=0}^n \sigma(0,\ldots,i) \otimes \sigma(i,\ldots n),\]
where $\sigma(0,\ldots,i)$ and $\sigma(i,\ldots n)$ denote the first $i$-dimensional face and last $(n-i)$-dimensional face, respectively, of $\sigma$. 

This construction does not define a categorical coalgebra since $\delta \colon C_1(X) \to C_0(X)$ may not vanish in general. However, we modify $\delta$ as follows. Let $e\colon C_*(X)\rightarrow R$ be the $1$-cochain induced by sending non-degenerate $1$-simplices to $1_R$ and degenerate $1$-simplices to $0$. Define 
\begin{align}\label{eq:modifiedDifferential} 
    d = \delta  - (\text{id} \otimes e - e\otimes \text{id}) \circ \Delta 
\end{align}
and
\begin{align}\label{eq:curvatureChains} 
    \eta = (e \otimes e) \circ \Delta + e \circ \partial.
\end{align}
The map $d \colon C_*(X) \to C_{*-1}(X)$ does not square to zero in general, but its failure to be a differential is controlled by the $2$-cocycle $\eta$ in the sense of \cref{eq:curvature}. Note that $\eta$ vanishes on $2$-simplices with non-degenerate faces. We denote $\mathcal{C}(X):=(C_*(X), d, \Delta, \eta)$. A straightforward check yields that the assignment $X \mapsto \mathcal{C}(X)$ defines a functor
\[ \mathcal{C} \colon \mathsf{sSet} \to \mathsf{cCoalg}_R.\]
We describe $\mathcal{C}(\triangle^n)$ more explicitly for $\triangle^n \in \mathsf{sSet}$ the standard $n$-simplex. In this case, we have a canonical basis for the normalized chains $C_*(\triangle^n)$ given by non-degenerate simplices of $\triangle^n$. On any such $\sigma \in (\triangle^n)_j$, we have 
\begin{align} \label{eq:removefirstandlast}
    d(\sigma)= \sum _{i=1}^{j-1} (-1)^i\delta_i(\sigma),
\end{align}
where $\delta_i \colon (\triangle^n)_j \to (\triangle^n)_{j-1}$ is the $i$-th face map. In other words, $d$ is given by removing the first and last faces from the usual simplicial boundary map. Also note that, in this example, any non-degenerate $2$-simplex has non-degenerate faces, so the curvature term is zero. The same argument yields that if $X$ is a simplicial set obtained from a simplicial complex through a coherent ordering of the vertices of each simplex and adding formal degeneracies (so that all simplices are uniquely determined by their vertices), then, on any non-degenerate simplex, the differential of $\mathcal{C}(X)$ is also given by \cref{eq:removefirstandlast}.

\subsection{Cobar construction}\label{subsec:cobar}
Any categorical coalgebra gives rise to a natural dg category through a functor 
\[ \Omega \colon \mathsf{cCoalg}_R \to \mathsf{dgCat}_R, \]
called the \textit{cobar construction}, which we now recall. Given a categorical coalgebra $C=(C,d, \Delta, \eta)$, the objects of $\Omega(C)$ are the elements of $\mathcal{S}_C$. For any two $a,b \in \mathcal{S}_C$ define
\[\Omega(C)(a,b)=
\bigoplus_{n=0}^{\infty} R_a \boxtimes (s^{-1}\overline{C})^{\boxtimes n} \boxtimes R_b, \]
where $\boxtimes$ denotes the cotensor product over the coalgebra $C_0$, for any $a \in \mathcal{S}_C$, $R_a$ is $R$ equipped with the $C_0$-bicomodule determined by the inclusion $\{a\} \hookrightarrow \mathcal{S}_C$, $\overline{C}:= C/C_0$, and $(s^{-1}\overline{C})^{\boxtimes 0}:= C_0$. The identity morphism $\id_a$ at $a \in \mathcal{S}_C$ corresponds to the unit of $R$ through the following identification  \[1_R \in R \cong R_a \cong R_a \boxtimes (s^{-1}\overline{C})^{\boxtimes 0} \boxtimes R_a \subseteq \Omega(C)(a,a)_0, \]
and extending this assignment $a \mapsto \id_a$ linearly gives us a map
\[ \gamma \colon C_0 \to \Omega(C). \]
Each differential 
\[ d_{\Omega C} \colon \Omega(C)(a,b) \to \Omega(C)(a,b)\]
is induced by the extending the sum $d+ \Delta + \eta$ as a derivation with appropriate degree shifts. The equation \[d_{\Omega C} \circ d_{\Omega C} =0\] is equivalent to the compatibility of $d$ and $\Delta$, the coassociativity of $\Delta$, and the curvature equation relating $\eta$, $d$, and $\Delta$. The composition of morphisms is given by concatenation of monomials.

For any simplicial set $X$, the dg category $\Omega \mathcal{C}(X)$ has as objects the elements of the set $X_0$ of vertices in $X$ and the chain complex of morphisms $\Omega \mathcal{C}(X)(v,w)$ between any  $v, w\in X_0$ is generated as a graded module by the identity $\id_v$ when $v = w$, together with symbols
\[ \{\sigma_1|\dots|\sigma_N\},\ \text{where $\sigma_i \in \bigcup_{j=1}^\infty X^{\text{nd}}_j$ with}\ s(\sigma_1)=v, t(\sigma_N) = w, \text{ and }s(\sigma_i) = t(\sigma_{i-1})\ \text{for}\ 1 \le i \le N \]
in homological degree $\sum_{i=1}^N \dim(\sigma_i) - N$. Above, $X^{\text{nd}}_j \subseteq X_j$ denotes the set of non-degenerate $j$-simplices in $X$ and \[s, t \colon \bigcup_{j=1}^\infty X_j \to X_0\] denote the first (source) and last (target) vertex maps, respectively. A sequence of simplices $\{\sigma_1 | \dots |\sigma_N\}$ in $X$ satisfying the above conditions is called an \textit{open necklace} in $X$ from $v$ to $w$ and the simplices $\sigma_1,\ldots,\sigma_N$ are called the \textit{beads} of the (open) necklace. 

The cobar construction models the dg category of paths in a space. More precisely, given any topological space $Z$, denote by $\mathcal{P}(Z)$ be the dg category whose objects are the points of $Z$ and, given $v,w \in Z$ we have that
$\mathcal{P}(Z)(v,w):= S_*(P_{v\to w}Z)$, the chain complex obtained by applying the normalized simplicial singular chains functor to the topological space $P_{v \to w}Z$ of Moore paths in $Z$ from $v$ to $w$. Composition is induced by concatenation of paths. This construction gives rise to a functor \[\mathcal{P} \colon \mathsf{Top} \to \mathsf{dgCat}_R.\]

\begin{theorem} \label{thm:cobarmodel}
If $X$ is a simplicial set whose homotopy category is a groupoid, then $\Omega \mathcal{C}(X)$ is naturally quasi-equivalent to $\mathcal{P}(|X|)$, the dg category of paths in the geometric realization of $X$.
\end{theorem}
We sketch the proof of the above theorem below. 

\subsubsection{A family of paths associated to an open necklace}\label{sec:familyOfPaths}
We now recall an explicit map realizing the quasi-equivalence in the above theorem. The main idea, originally due to Adams \cite{adams1956cobar}, is to associate to any open necklace $\{\sigma_1|\dots|\sigma_N\}$ in $X$ a family of paths in $|X|$ from $s(\sigma_1)$ to $t(\sigma_N)$ parametrized by a cube of dimension $\sum_{i=1}^N \dim(\sigma_i)-N$, in a way that is compatible with the cobar differential. 

Denoting by $v_0, v_1, \ldots, v_n$ the vertices of $\triangle^n$, there is a natural isomorphism of chain complexes
\[\Omega \mathcal{C}(\triangle^n)(v_0,v_n) \cong C^{\Box}_*(\Box^{n-1}),\] where the right hand side denotes the (normalized) cubical chain complex on the standard cubical complex structure of the $(n-1)$-cube. This isomorphism is determined by sending a sequence $\{\sigma_1 | \cdots |\sigma_N\} \in \Omega(\triangle^n)(v_0,v_n)_0$, where $\sigma_i$ is a $1$-simplex in $\triangle^n$, to the $0$-cube (vertex) in $\Box^{n-1}$ determined by the string $(\varepsilon_1,\ldots, \varepsilon_{n-1})$ where $\varepsilon_j=1$ if $v_j \in \{ s(\sigma_2), s(\sigma_2),\ldots, s(\sigma_N)\}$ and $\varepsilon_j=0$ otherwise. 

More generally, for any simplicial set $X$ and $v,w \in X_0$ there is a cubical set $P^\Box_{v \to w}X$ together with a natural isomorphism
\[ \Omega \mathcal{C}(X)(v,w) \cong C^\Box_*(P^{\Box}_{v \to w}X).\] The $k$-cubes in $P^\Box_{v \to w}X$ are labeled by open necklaces $\{\sigma_1|\dots|\sigma_N\}$ in $X$
with $k= \sum_{i=1}^N \dim(\sigma_i) - N$, $s(\sigma_1)=v$, $t(\sigma_N)=w$; and these cubes are glued according to how open necklaces from $v$ to $w$ fit together inside $X$. The boundary of a cube labeled by an open necklace exactly corresponds to applying the cobar differential $d_{\Omega C}$. 

The above interpretation of the morphisms in the cobar construction as the cubical chains on a cubical set may be used to construct a quasi-equivalence of dg categories
\[\mathcal{a}_X \colon \Omega \mathcal{C}(X) \to \mathcal{P}(|X|)\]
as we now explain.  We recall the construction, originally due to Adams, of a continuous map
\[ a \colon |P^\Box_{v \to w}X|_\Box \to P_{v\to w}|X| \]
from the (cubical) geometric realization to the space of continuous paths in the (simplicial) geometric realization. Parametrize the topological simplex by
\[ \triangle^k = \{(t_1,\dots,t_k)\ |\ 0 \le t_1 \le \dots \le t_k \le 1\} \subset \mathbb{R}^k \]
noting that in this parametrization the vertex $(i)$ of $\triangle^k$ has coordinates $(0,\dots,\underset{i}{0},1,\dots,1)$. We have a continuous map
\begin{align}\label{eq:cubeToSimplex}
    q \colon \Box^k &\to \triangle^k\\
    (r_1,\dots,r_k) &\mapsto (r_1 r_2 \dots r_k, r_2r_3 \dots r_k,\dots,r_{k-1}r_k,r_k)
\end{align}
which restricts to a homeomorphism in the interior. Denote
\[ E^{k-1} = \{(p_1,\dots,p_{k-1},t)\ | (p_1,\dots, p_{k-1}) \in \Box^{k-1}, 0 \le t \le (\sum_{i=1}^{k-1} p_i) + 1\} \subset \mathbb{R}^{k}.\]
Define a continuous piecewise affine map $E^{k-1} \to \Box^k$ by 
\begin{align*}
    (p_1,\dots,p_{k-1},t) \mapsto \begin{cases}
        (1-t,1,1\dots,1), & 0 \le t \le p_1 \\
        (1-p_1, 1-t+p_1,1,\dots,1) & p_1 \le t \le p_1 + p_2 \\
        (1-p_1, 1-p_2, 1-t+p_1+p_2,1,\dots,1) & p_1 + p_2 \le t \le p_1 + p_2 + p_3 \\
        \dots & \dots \\
        (1-p_1,1-p_2,\dots,1-t + p_1+\dots+ p_{k-1}) & \sum p_i \le t \le \sum p_i + 1
    \end{cases} 
\end{align*}
Composing with the map $q$, we obtain a map $E^{k-1} \to \triangle^k$ which sends
\[ (p_1,\dots,p_{k-1},0) \mapsto (1,1,\dots,1) = (0), \quad (p_1,\dots,p_{k-1},1+\sum_{i=1}^{k-1}p_i) \mapsto (0,0,\dots,0) = (k).\]
This associates to any line $\{(p_1,\dots,p_{k-1},t)\}$ with fixed $(p_1,\dots,p_{k-1})$ and varying $t$ a Moore path in $\triangle^k$ from the vertex $(0)$ to the vertex $(k)$, giving a continuous map
\begin{align}\label{eq:adamstw}
    \alpha \colon \Box^{k-1} \to P_{(0)\to(k)} \triangle^k.
\end{align}
The map $\alpha$ induces a map 
\begin{align} \label{eq:adamsgeometric}
    a_X \colon |P^\Box_{v \to w}X|_\Box \to P_{v\to w}|X|
\end{align} 
as follows. Consider a cube $\Box^{|\sigma_1|-1} \times \cdots \times \Box^{|\sigma_N|-1}$ in the geometric realization $|P^\Box_{v \to w}X|$ labeled by an open necklace $\{\sigma_1 | \cdots | \sigma_N\}$ in $X$. On any point $u=(u_1,\ldots, u_N)$ in this cube (with $u_i \in \Box^{|\sigma_i|-1}$) define
\[ a_X(u)= \big(P{\sigma_1} \circ \alpha(u_1)\big) * \cdots * \big(P{\sigma_N} \circ \alpha(u_N)\big),\]
where $P{\sigma} \colon P_{(0) \to (k)}\triangle^k \to P_{s(\sigma) \to t(\sigma)}|X|$
denotes the map on path spaces determined by a simplex $\sigma \in X_k$. 
\begin{theorem}\label{thm:homotopicWhenGroupoidPaths}
    If $X$ is a simplicial set whose homotopy category is a groupoid then 
    \[ a_X \colon |P^\Box_{v \to w}X|_\Box \to P_{v\to w}(|X|) \] 
    is a weak homotopy equivalence.
\end{theorem}
\begin{proof}
The theorem follows by the same argument as the proof of \cite[Theorem 1]{RS1}. See also \cite[Corollary 4.2]{CHL}. We provide a sketch for completeness. Without loss of generality, we may assume that $X$ is connected and $v=w$ so that
$P_{v\to v}(|X|)$ is the space of Moore loops based at $v \in |X|$. The map $a_X$ fits into the following commutative diagram:
\[\begin{tikzcd}
	{|P^\Box_{v \to v}X|_\Box } & {P_{v\to v}(|X|)} \\
	{|P^\Box_{v }X|_\Box } & {P_{v}(|X|)} \\
	{|X|} & {|X|}
	\arrow["{a_X}", from=1-1, to=1-2]
	\arrow[from=1-1, to=2-1]
	\arrow[from=1-2, to=2-2]
	\arrow["{a'_X}", from=2-1, to=2-2]
	\arrow[from=2-1, to=3-1]
	\arrow[from=2-2, to=3-2]
	\arrow["{\text{id}}", from=3-1, to=3-2]
\end{tikzcd}\]
In the above diagram, $P_v(|X|)$ denotes the contractible space of Moore paths based at $v$ with free endpoint, $|P^\Box_{v }X|_\Box$ is constructed in a similar way to $|P^\Box_{v \to w}X|_\Box$ with cubes labeled by necklaces but removing the condition that the last bead of a necklace must have vertex $w$ as a target; namely, via necklaces with ``free" last vertex. One can then construct a continuous map
$a'_X \colon |P^\Box_vX|_\Box \to P_v(|X|)$ in a similar way to $a_X$ making the top square commute, where the two top vertical maps are natural inclusions. The space $|P^\Box_vX|_\Box$ can be shown to be contractible and, consequently, $a'_X$ is a weak homotopy equivalence. The bottom vertical maps are natural projections given by evaluating on the free endpoint of a path (interpreted appropriately on the left vertical map) making the bottom square commute. The sequence of vertical maps in the right hand side of the diagram define the path fibration with fiber the based loop space. The sequence of maps on the left hand side define a quasi-fibration as can be checked by applying the classical Dold-Thom criterion. Passing to long exact sequences on homotopy groups and using the $5$-lemma yields that $a_X$ is a weak homotopy equivalence. \end{proof}

\noindent \textit{Proof of \cref{thm:cobarmodel}: } After passing to chains, the map $a_X$ above induces a quasi-isomorphism
 \[\mathcal{a}_{X}^{v\to w} \colon \Omega \mathcal{C}(X)(v,w)\cong C^\Box_*(P^\Box_{v \to w}X) \xrightarrow{i_\Box} S^{\Box}_*(|P^\Box_{v \to w}X|_\Box) \xrightarrow{S^\Box_*(a)} S^\Box_*( P_{v\to w}(|X|))\xrightarrow{Sd_\Box} S_*( P_{v\to w}(|X|)),\]
where $i_\Box$ is the natural map from the cubical chains of a cubical set to the singular cubical chains of its geometric realization, $S_*^{\Box}$ denotes the (normalized) singular cubical chains functor, and $Sd_\Box$ a subdivision map from singular cubical chains to singular simplicial singular chains. All of the maps $\mathcal{a}_{X}^{v \to w}$ together determine the desired quasi-equivalence of dg categories
\[ \mathcal{a}_X \colon \Omega \mathcal{C}(X) \xrightarrow{\simeq} \mathcal{P}(|X|). \]\qed

\subsection{CoHochschild complex}

For any categorical coalgebra $C=(C,d_C,\Delta,\eta)$ the \textit{coHochschild complex of $X$} is defined as the chain complex \[(\mathcal{coCH}_*(C), \partial) :=C\underset{C_0^e}{\boxtimes} \Omega C,\] 
where the dg category $\Omega C$ is considered as the dg $C_0$-bicomodule $\bigoplus_{v,w \in \mathcal{S}_C} \Omega C(v,w)$ with $C_0$-coactions induced by the source and target maps. Explicitly, $\mathcal{coCH}_k(C)$ is generated by monomials $x=x_0\{x_1 | \dots |x_N\}$, where 
\begin{itemize}
\item $x_i \in C_{n_i}$ and $N \geq 0$ is a non-negative integer

    \item $n_0 \geq 0$ and $n_i>0$ for $i=1,\ldots,N$,  
    \item $n_0 + n_1+ \cdots+ n_N-N=k$, and
    \item $x_0 \otimes x_1$ and $x_N \otimes x_0$ are elements in the cotensor product $C \underset{C_0}{\boxtimes} C \subset C \otimes_R C$. 
\end{itemize}
The differential
\[\partial \colon \mathcal{coCH}_k(C) \to \mathcal{coCH}_{k-1}(C)\]
is defined by
\[\partial= d_C \boxtimes \text{id}_{\Omega C} + \text{id}_C \boxtimes d_{\Omega C}+ \tau,\]
where
\begin{align}\label{eq:coHochDifferential}
    \tau(x_0\{x_1 | \dots |x_N\})= &-(-1)^{x^{(1)}_0} x_0^{(1)}\{x_0^{(2)}|x_1 | \dots |x_N\} \\
    &+(-1)^{(x_0^{(1)} + 1)(x_0^{(2)} + \sum_{i=1}^N (x_i+1))} x_0^{(2)}\{x_1 | \dots |x_N|x_0^{(1)}\}
\end{align}
writing $\Delta(x)= x^{(1)} \otimes x^{(2)}$. This construction gives rise to a functor
\[ \mathcal{coCH}_* \colon \mathsf{cCoalg}_R \to \mathsf{Ch}_R, \]
which is homotopical in the following sense: if a morphism $f \colon C \to C'$ of a categorical coalgebras induces a quasi-equivalence of dg categories $\Omega(f) \colon \Omega(C) \to \Omega(C')$, then $\mathcal{coCH}_*(f) \colon \mathcal{coCH}_*(C) \to \mathcal{coCH}_*(C')$ is a quasi-isomorphism of chain complexes.  

For any simplicial set $X$,  the complex $\mathcal{coCH}_k(\mathcal{C}(X))$
is generated by sequences 
\[ \sigma_0\{\sigma_1|\dots|\sigma_N\} \]
of non-degenerate simplices in $X$, where $\sigma_0$ is of dimension $\geq 0$, each $\sigma_{i\neq 0}$ is of dimension $\ge 1$, $t(\sigma_0)= s(\sigma_1)$, $t(\sigma_i)=s(\sigma_{i+1})$ for $i=1,\ldots,N-1$, $t(\sigma_N)=s(\sigma_0)$, and $|\sigma_0| + |\sigma_1| + \cdots |\sigma_N|-N=k$. We call a sequence as above a \emph{closed necklace} in $X$ with \textit{beads} $\sigma_1,\ldots, \sigma_N$ and \textit{marked bead} $\sigma_0$.

\begin{theorem}\label{thm:cohochmodel} If $X$ is a simplicial set whose homotopy category is a groupoid, then $\mathcal{coCH}_*(\mathcal{C}(X))$ is naturally quasi-isomorphic to to $S_*(L|X|)$, the singular chains on the free loop space of the geometric realization of $X$.
\end{theorem}
We sketch the proof of the above theorem below. 

\subsubsection{A family of free loops associated to a closed necklace} As in the case of path spaces, we may realize the equivalence in \cref{thm:cohochmodel} by an explicit quasi-isomorphism as we now explain.  The main idea is to associate to any closed necklace $\sigma_0\{\sigma_1|\dots|\sigma_N\}$ in $X$ a family of free loops in $|X|_{\triangle}$ parametrized by an appropriate subdivision of a cube of dimension $|\sigma_0| + |\sigma_1| + \cdots + |\sigma_N|-N$ compatible with the coHochschild differential. 

For any simplicial set $X$ one may construct an abstract cell complex $L^\boxright X$ as follows:
for any closed necklace $\sigma_0\{\sigma_1 | \cdots |\sigma_N\}$ in $X$ we consider a cell of the form 
\[ \boxright^{|\sigma_0|} \times \Box^{|\sigma_1| + \cdots + |\sigma_N|-N}\]
where $\Box^l$ denotes the $l$-cube and $\boxright^k$ is a convex polytope coined as the $k$-\textit{freehedron} whose underlying combinatorial type may be obtained by suitably subdividing one of the faces of the $k$-cube $\Box^k$, see \cite{RS2, RT} \footnote{This subdivision reflects the fact that for the marked bead of a closed necklace, which is allowed to be of dimension $0$, there is an additional parameter encoding the (varying) basepoint of each loop in the associated family parametrized by a cube.}. These cells are glued according to how closed necklaces fit inside $X$ (via simplicial face and degeneracy maps). The crucial property of $L^\boxright X$ is that after taking the corresponding cellular chain complex we exactly obtain the coHochschild complex of the categorical coalgebra of chains; namely, there is an isomorphism of chain complexes
\[ \coCH_*(\mathcal{C}(X)) \cong C_*^\boxright(L^\boxright X). \]
Denote by $|L^\boxright X|_{\boxright}$ the topological space obtained by geometrically realizing the abstract freehedral complex $L^\boxright X$. The above interpretation of the coHochchild complex as cellular chain complex may be used to construct a quasi-isomorphism
\[\ell_X \colon \coCH_*(\mathcal{C}(X)) \to  S_*(L|X|)\]
as follows. Consider the map $\alpha \colon \Box^{k-1} \to P_{(0) \to (k)}\triangle^k$ from \cref{eq:adamstw} and write $\alpha(u) \colon [0,r_{\alpha(u)}] \to \triangle^k$ for any $u \in \Box^{k-1}$. Define a map
\begin{align}\label{eq:adamsgeometricloops}
    l_X \colon |L^{\boxright} X|_{\boxright} \to L|X| 
\end{align}
by sending a point $p=(t,u_0, u_1, \dots, u_N)\in |L^{\boxright} X|_{\boxright}$ in a cell $\boxright^{|\sigma_0|} \times \Box^{|\sigma_1|-1} \times \cdots \times \Box^{|\sigma_N|-1}$ (with $t \in \Box^1$ and $u_i \in \Box^{|\sigma_i|-1}$) labeled by a closed necklace $\sigma_0\{\sigma_1 | \cdots |\sigma_N\}$ in $X$ to the Moore loop determined by the concatenation of paths 
\[l_X(p)=\big(P{\sigma_0} \circ \alpha(u_0)|_{[r_{\alpha(u_0)}t, r_{\alpha(u_0)}]}\big) *  \big(P{\sigma_1} \circ \alpha(u_1)\big)* \cdots * \big(P{\sigma_N} \circ \alpha(u_N)\big)*\big(P{\sigma_0} \circ \alpha(u_0)|_{[0,r_{\alpha(u_0)}t]}\big).\]

\begin{theorem}\label{thm:homotopicWhenGroupoidLoops}
    If $X$ is a simplicial set whose homotopy category is a groupoid, then \[l_X \colon |L^{\boxright} X|_{\boxright} \to L|X|\] is a weak homotopy equivalence. 
\end{theorem}
\begin{proof} The proof is analogous to that of \cref{thm:homotopicWhenGroupoidPaths} but now comparing two quasi-fibrations modeling the free loop fibration, with the map $l_X$ between total spaces, and the map $a_X$ between  fibers, see \cite{RS2}. 
\end{proof}

\noindent \textit{Proof of \cref{thm:cohochmodel}: } After passing to chains, the map $l_X$ above induces the desired quasi-isomorphism
\[ \ell_X \colon \coCH_*(\mathcal{C}(X)) \cong C_*^\boxright(L^\boxright X) \xrightarrow{i_{\boxright}} S_*^\boxright(|L^\boxright X|_\boxright) \xrightarrow{S_*^\boxright(l)} S_*^\boxright( L|X|) \xrightarrow{Sd_{\boxright}} S_*(L|X|),\]
where $i_\boxright$ is the natural map from the freehedral chains of a freehedral complex to the singular freehedral chains of its geometric realization, $S_*^{\boxright}$ denotes the (normalized) singular freehedral chains functor, and $Sd_\boxright$ a subdivision map from singular freehedral chains to singular simplicial singular chains (whose existence may be deduced from the Acyclic Models Theorem). 
\qed

\section{Inverting edges}\label{sec:invertingedges}
We describe a convenient way of replacing any simplicial set, up to homotopy equivalence, by one whose homotopy category is a groupoid. This construction will be used to present the underlying homotopy type of a triangulated manifold.

Let $\mathcal{J}$ be the groupoid with two objects $j_0$ and $j_1$ and exactly two non-identity morphisms $e \colon j_0 \to j_1$ and $e^{-1} \colon j_1 \to j_0$. Consider the simplicial set $J= \text{Nerve}(\mathcal{J})$. This is a fibrant replacement for the standard $1$-simplex $\triangle^1$, namely, there is a homotopy equivalence $\triangle^1 \hookrightarrow J$ and $J$ is a Kan complex. We may label the $k$-simplices of $J$ by binary sequences
\[ (\varepsilon_0  \dots \varepsilon_k),\quad  \varepsilon_i = 0\ \text{or}\ 1, \]
with nondegenerate simplices corresponding to alternating sequences. Thus there are exactly two non-degenerate $k$-simplices for each $k$. The geometric realization of $J$ is a model for the contractible infinite dimensional sphere $S^{\infty}$. Given a simplicial set $X$, we invert the 1-simplices of $X$ by attaching a copy of $J$ along each 1-simplex. In other words, define a simplicial set $\widetilde{X}$ by the pushout diagram
\[
\xymatrix{
    (\triangle^1)^{X_1} \ar[r] \ar[d] & J^{X_1} \ar[d] \\
    X \ar[r]^i & \widetilde{X}.
}
\]
\begin{proposition}
For any simplicial set $X$, the homotopy category of $\widetilde{X}$ is a groupoid and the natural inclusion $i \colon X \hookrightarrow \widetilde{X}$ is an acyclic cofibration. In particular, the induced map $|i| \colon |X| \hookrightarrow |\widetilde{X}|$ is a homotopy equivalence.
\end{proposition}
\begin{proof}
  This follows from the fact that $\triangle^1 \hookrightarrow J$ is a Kan replacement. 
\end{proof}

We now describe an explicit homotopy inverse for $|i| \colon |X| \hookrightarrow |\widetilde{X}|$. For each binary sequence $(\varepsilon_0  \dots \varepsilon_k)$, define a continuous map
\begin{align*}
    f_{(\varepsilon_0  \dots \varepsilon_k)} \colon \quad \triangle^k &\to \triangle^1 \\
    (t_1,\dots,t_k) &\mapsto \left( \frac{1}{2}(1+\sum_{i=1}^k (-1)^k(t_{i+1}-t_i)), \frac{1}{2}(1-\sum_{i=1}^k (-1)^k(t_{i+1}-t_i)) \right),
\end{align*}
where we set $t_0 = 0, t_{k+1} = 1$ for convenience of notation. One can check that applying $f_{(\varepsilon_0  \dots \varepsilon_k)}$ to the simplex in $J$ labeled by $(\varepsilon_0  \dots \varepsilon_k)$ gives a continuous map $f \colon |J| \to \triangle^1$ to the interval. Putting all these maps together for each 1-simplex of $X$ gives a natural map 
\begin{align}\label{eq:fX}
    f_X \colon |\widetilde{X}| \to |X|.
\end{align}
\begin{proposition}\label{prop:deformationRetract}
    There is a homotopy making $f_X$ a deformation retract for the inclusion $|X| \hookrightarrow |\widetilde{X}|$.
\end{proposition}
\begin{proof}
    By the definition of $\widetilde{X}$ as obtained from $X$ by gluing copies of $J$ along the 1-skeleton, it is sufficient to prove this for $X = \triangle^1$, in which case this homotopy can be constructed inductively on the dimension of cells by using the homotopy extension property of pairs of CW complexes.
\end{proof}

We described natural maps
\[ a_Y \colon |P^\Box_{v \to w}Y|_\Box \to P_{v\to w}(|Y|), \quad  l_Y \colon |L^\boxright Y|_\boxright \to L(|Y|) \]
for any simplicial set $Y$ in \cref{eq:adamsgeometric} and \cref{eq:adamsgeometricloops}, respectively. Take $Y = \widetilde{X}$ and define
\[ \mathcal{A}_X \colon |P^\Box_{v \to w}\widetilde{X}|_\Box \to P_{v\to w}(|X|), \quad  \mathcal{L}_X \colon |L^\boxright \widetilde{X}|_\boxright \to L(|X|)  \]
 by composing the maps $a_{\widetilde{X}}$ and $l_{\widetilde{X}}$ with the maps induced by $f_X$.
\begin{proposition}\label{prop:homotopyEquivalence}
    The maps $\mathcal{A}_X$ and $\mathcal{L}_X$ are weak homotopy equivalences.
\end{proposition}
\begin{proof}
Since the homotopy category of $\widetilde{X}$ is a groupoid, it follows from \cref{thm:homotopicWhenGroupoidPaths} and \cref{thm:homotopicWhenGroupoidLoops} that $a_{\widetilde{X}}$ and $l_{\widetilde{X}}$ are weak homotopy equivalences, respectively. Since the maps induced by $f_X$ at the level of path and loop spaces are homotopy equivalences, the result follows. 
\end{proof}

For any simplicial set $X$ we label the new simplices in $\widetilde{X}$ through the identification
\[ x^k_\sigma \to (\overbrace{0101\dots}^{k+1})_\sigma, \quad y^k_\sigma \to (\overbrace{1010\dots}^{k+1})_\sigma, \]
for any non-degenerate $\sigma \in X_1$. We have $x_\sigma^1 = \sigma$ and, for convenience, we write $y_\sigma^1 = \check\sigma$.  The underlying graded module of $\mathcal{C}(\widetilde{X})$ is then given by 
    \[ C_*(X) \oplus \bigoplus_{\sigma \in C^{\mathrm{nd}}_1} R \cdot \check\sigma \oplus \bigoplus_{k \ge 2} \left(R \cdot x_{\sigma}^k \oplus R \cdot y_{\sigma}^k \right)\]
 and the Alexander-Whitney coproduct is given on the new basis elements by 
    \[ \Delta(\check\sigma) = t(\sigma) \otimes \check\sigma + \check\sigma \otimes s(\sigma), \]
    and  
    \begin{align*}
        \Delta(x_\sigma^k) &= +\sum_{0 \le i \le k/2 - 1} x^{2i+1}_\sigma \otimes y^{k-2i-1}_\sigma + \sum_{1 \le i \le (k-1)/2} x^{2i}_\sigma \otimes x^{k-2i}_\sigma + \begin{cases}
            s(\sigma) \otimes x_\sigma^k + x_\sigma^k \otimes t(\sigma), &k \text{\ odd} \\
            s(\sigma) \otimes x_\sigma^k + x_\sigma^k \otimes s(\sigma), &k \text{\ even}
        \end{cases}\\
        \Delta(y_\sigma^k) &= -\sum_{0 \le i \le k/2 - 1} y^{2i+1}_\sigma \otimes x^{k-2i-1}_\sigma - \sum_{1 \le i \le (k-1)/2} y^{2i}_\sigma \otimes y^{k-2i}_\sigma + \begin{cases}
            t(\sigma) \otimes y_\sigma^k + y_\sigma^k \otimes s(\sigma), &k \text{\ odd} \\
            t(\sigma) \otimes y_\sigma^k + y_\sigma^k \otimes t(\sigma), &k \text{\ even}.
        \end{cases}
    \end{align*}

Note that, when applied to $\{x^2_\sigma\}$ and $\{y^2_\sigma\}$, the cobar differential has a non-zero contribution from the curvature term $\eta$ defined in \cref{eq:curvatureChains} giving
\[ d_{\Omega C} \{x^2_\sigma\} = \{\sigma | \check\sigma\} - \text{id}_{s(\sigma)}, \quad d_{\Omega C}\{y^2_\sigma\} = -\{\check\sigma | \sigma\} + \text{id}_{t(\sigma)}.\]
\begin{remark}
    In general, the composition of functors $\coCH_* \circ \mathcal{C} \colon \mathsf{sSet} \to \mathsf{Ch}_R$ does not send weak homotopy equivalences of simplicial sets to quasi-isomorphisms of chain complexes. However, the functor $\coCH_* \circ \mathcal{C} \circ (\widetilde{ - })$, obtained by pre-composing with the above construction, does send weak homotopy equivalences to quasi-isomorphisms.
\end{remark}

\section{Constant loops}\label{sec:constantloops}
For any topological space $Y$, denote by 
\[\rho_Y \colon Y \hookrightarrow LY\]
the natural embedding that sends a point $y \in Y$ to the constant loop $c_y \colon \{0\} \to Y$, $c_y(0)=y$. This induces a map on homology $H_*(\rho_Y)\colon H_*(Y) \to H_*(LY)$. We shall consider any natural chain-level lift of $H_*(\rho_Y)$ to the coHochschild complex as given by the following. 

\begin{proposition}\label{prop:constantLoops}
    For any simplicial set $X$, there is a natural chain map
    \[ \iota_X \colon C_*(X) \to \coCH_*(\mathcal{C}(\widetilde{X})), \]
    and a natural chain homotopy between $S_*(Lf_X) \circ \ell_{\widetilde{X}} \circ \iota_X$ and $S_*(\rho_{|X|}) \circ i$, where $i \colon C_*(X) \hookrightarrow S_*(|X|)$ is the natural quasi-isomorphism from simplicial chains to singular chains. In particular, we have $H_*(\iota_X)=H_*(\rho_{|X|})$. 
\end{proposition}

\begin{proof} 
    First define $\iota_X \colon C_0(X) \to \coCH_0(\mathcal{C}(\widetilde{X}))$ by declaring 
    \[ \iota_X(v):=v \cdot \text{id}_v \in \coCH_0(\mathcal{C}(\widetilde{X}))\]
    for any vertex $v \in X_0$. This natural map is well defined on $0$-th homology: if $\sigma \in X_1$ is a $1$-simplex with $s(\sigma)=v$ and $t(\sigma)=w$ so that $\partial(\sigma)=w-v$,
    then setting 
    \[\iota_X(\sigma):= \sigma \{ \check\sigma\} + v \{x_{\sigma}^2\} + w\{y_{\sigma}^2\} \in \coCH_1(\mathcal{C}(\widetilde{X}))\]
    yields
    \[ \partial(\iota_X(\sigma))= w \cdot \text{id}_w - v \cdot\text{id}_v. \]
    The proposition now follows from the Acyclic Models Theorem of \cite{EMcL}, since the functors $\mathsf{sSet} \to \mathsf{Ch}_R$ determined by $X \mapsto C_*(X)$ and $X \mapsto \coCH_*(\mathcal{C}(\widetilde{X}))$ are representable and acyclic, respectively, on the set of standard simplices.
\end{proof}

The method of acyclic models, as applied in the above proof, relies on naturality to determine $\iota_X$ from all the maps $\iota_{\triangle^n}$. This implies that for any non-degenerate $n$-simplex $\sigma \in X_n$, the coHochschild chain $\iota_X(\sigma)$ is \textit{supported} on the single simplex $\sigma$; i.e., through the map $S_*(Lf_X) \circ \ell_{\widetilde{X}}$, we may think of $\iota_X(\sigma)$ as an $n$-dimensional family of loops in $|X|$ all lying inside the subset $\sigma \subseteq |X|$.

\begin{definition}\label{def:support}
    Let $X$ be a simplicial set. The \textit{support} of an element
    \[ c = \sum_{i \in I} c_i \sigma_i, c_i \neq 0 \quad \text{for all}\ i \]
    in $C_*(X)$ is the subset of $|X|$ given by 
    \[ \mathrm{Supp}(c) = \bigcup_{i \in I} \sigma_i. \] 
    We extend this notion to $C_*(\widetilde{X})$ by setting
    \[ \mathrm{Supp}(x^k_\sigma) = \mathrm{Supp}(y^k_\sigma) = \sigma\]
    for all $k$ and non-degenerate $\sigma \in X_1$. We further extend the notion of support to $\coCH_*(\mathcal{C}(\widetilde{X}))$ by setting
    \[\mathrm{Supp}(z_0\{z_1| \cdots |z_N\})=\bigcup_{i=0}^N \mathrm{Supp}(z_i)  \subseteq |X|,\]
    where $z_i \in C_*(\widetilde{X})$.
\end{definition}

The following is immediate from definitions. 
\begin{proposition} 
For any $\sigma \in X_n$, $\mathrm{Supp}(\iota_X(\sigma))= \sigma \subseteq |X|$. Furthermore, for any $z \in \coCH_*(\mathcal{C}(\widetilde{X}))$ we have that
   \begin{enumerate}
       \item  $\mathrm{Supp}(\partial z) \subset \mathrm{Supp}(z)$, 
       \item the quasi-isomorphism 
       \[S_*(Lf_X) \circ \ell_{\widetilde{X}} \colon \coCH_*(\mathcal{C}(\widetilde{X}))\to S_*(L|X|)\] sends $z$ to a chain of loops supported within $\mathrm{Supp}(z)$.
       \end{enumerate}
\end{proposition}

For any generator $x=\sigma_0\{ \sigma_1 | \sigma_1 | \dots |\sigma_N\} \in \coCH_*(\mathcal{C}(\widetilde{X}))$, $(S_*(Lf_X) \circ \ell_{\widetilde{X}})(x) $ consists of certain family of piecewise linear loops contained inside $\bigcup_{i=0}^N \sigma_i \subset |X|$. It will be convenient to enlarge these cells of loops inside of the loop space.
\begin{definition}\label{def:piecewiseLinearCells}
    Given a sequence $(\sigma_0,\sigma_1,\dots,\sigma_N)$ of simplices of $X$ such that $\sigma_i \cap \sigma_{i+1} \neq 0$ for all $0 \le i \le N-1$ and $\sigma_N \cap \sigma_0 \neq 0$, define $S_{\sigma_0,\sigma_1,\dots,\sigma_N}$ to be the subset of $L|X|$ of all Moore loops $(\ell,T)$ for which there are times 
    \[ 0 \le t_1 \le \dots \le t_N \le t_{i+1} = t_0 \le T \] 
    such that  $\ell|_{[t_i,t_{i+1}]}$ is piecewise linear in the `cubical coordinates', that is, the image of a piecewise linear path under the map $q \colon \Box^{\dim(\sigma_i)} \to \sigma_i$ given in \cref{eq:cubeToSimplex}, and contained in $\sigma_i$ for all $0 \le i \le N$.
\end{definition}

Note that for any generator $z=\sigma_0\{ \sigma_1 | \sigma_1 | \dots |\sigma_N\} \in \coCH_*(\mathcal{C}(\widetilde{X}))$, $S_{\sigma_0, \sigma_1, \dots, \sigma_N}$ is a contractible subset of $L|X|$ containing all loops in the family determined by $(S_*(Lf_X) \circ \ell_{\widetilde{X}})(z)$.

\subsection{Fineness and locality}
It will be important for us to control the locality of supports in the context of simplicial complexes. For that, we introduce some concepts that will allow us to quantify the `fineness' of a simplicial complex. We will consider simplicial subcomplexes of a simplicial complex $K$, and for ease of notation identify them with their image in the geometric realization.

Denote by $K_0$ the set of vertices and let $v,w\in K_0$. We say that a sequence $p = (\sigma_1,\dots,\sigma_N)$ of 1-simplices in $K$ \textit{links} $(v,w)$ if $v \in \sigma_1, w \in \sigma_N$ and $\sigma_i \cap \sigma_{i+1} \neq \varnothing$ for all $i$. In that case we say that this sequence has length $\mathrm{len}(p) = N$.
\begin{definition}
    The diameter of a sub-simplicial complex $A \subset K$ is defined by
    \[ \diam(A) = \max_{v,w \in A_0} \min_{\ \text{ $p$ links\ }(v,w)} \mathrm{len}(p) \]
\end{definition}
For example, the diameter of a single simplex is zero if it is of dimension zero and one otherwise. 
\begin{definition}\label{def:fine}
    Let $m$ be a positive integer. A simplicial complex $K$ is \emph{$m$-fine} if for every sub-simplicial complex $A \subset K$ with with $\diam(A) \le m$ we can choose a contractible sub-simplicial complex $Z_A$ containing $A$ in such a way that $Z_{A} \subset Z_{B}$ whenever $A \subset B$.
\end{definition}

Checking the condition above in practice seems to require a lot of guesswork, namely it requires to specify a $Z_A$ for each sub-simplicial complex $A\subset K$, but it turns out that this condition can always be satisfied by iterating barycentric subdivisions.
\begin{proposition}
    For any simplicial complex $K$ and any positive integer $m$, there exists a positive integer $k$ such that the $k$-times iterated barycentric subdivision $K^{(k)}$ is $m$-fine.
\end{proposition}
\begin{proof}
    For the first barycentric subdivision $K'$ of $K$, the stars of simplices $\tau \in K'$ satisfy the property that intersections
    \[ \mathrm{Star}(\tau) \cap \mathrm{Star}(\tau') \]
    are either empty or equal to $\mathrm{Star}(\rho)$ for some $\rho \in K'$. Let us now take $k$ such that $m+1 \le 2^{k-1}$. 
    
    We argue that any sub-simplicial complex $A \subset K^{(k)}$ of diameter $\le m$ is contained in the star of a simplex in $K'$. For that, note that a sub-simplicial complex that is \emph{not} contained inside of a single star of $K'$ must have two vertices that are at distance at least $2^{k-1}$ of each other, and therefore this subset must have diameter $\ge 2^{k-1}$. Conversely every subset $Y$ of diameter $\le 2^{k-1}-1$ is contained inside of some star.

    We now take $Z_A$ to be the smallest star containing $A$; this is well-defined since the intersection of any two stars in $K'$ is either empty or a star. Moreover, if $A \subset B$, $Z_{A} \cap Z_{B} \neq \varnothing$ so there is some simplex $\rho \in K'$ such that $\mathrm{Star}(\rho) = Z_{A} \cap Z_{B}$, which would contradict minimality unless $Z_{A} \subset Z_{B}$.
\end{proof}

Using the notion of support from \cref{def:support}, we can quantify the size of generators of the coHochschild complex.
\begin{definition}
    Let $K$ be a simplicial complex, $X$ a simplicial set obtained by any coherent ordering of the vertices of each simplex of $K$, and $m$ a positive integer. The chain complex of $m$-\textit{local coHochschild chains}, denoted by 
    \[  (\coCH_*(\mathcal{C}(\widetilde{X})))_{m\mathrm{-loc}},\]
    is the sub-chain complex of $\coCH_*(\mathcal{C}(\widetilde{X}))$ spanned by the generators whose support in $|X|\cong|K|$ have diameter $\le m$.
\end{definition}
 
By definition, the map $\iota_{X}$ from \cref{prop:constantLoops} lands within the sub-chain complex of $1$-local coHochschild chains. When $K$ is $m$-fine, all the sub-chain complexes of $m'$-local Hochschild chains, for $m' \le m$, are equally good representatives of the locus of constant loops, by the following proposition.
\begin{proposition}\label{prop:mLocalEquivalences}
    If $K$ is an $m$-fine simplicial complex and $X$ is the simplicial set determined by a coherent ordering of the vertices of $K$, then the inclusions
    \[ \Imm(\iota_X) \subseteq (\coCH_*(\mathcal{C}(\widetilde{X})))_{1\mathrm{-loc}} \subseteq \dots \subseteq (\coCH_*(\mathcal{C}(\widetilde{X})))_{m\mathrm{-loc}} \]
    are all chain homotopy equivalences.
\end{proposition} 
\begin{proof}
    We will show a construction of the first chain equivalence, as all the following ones are constructed similarly. For any $1$-fine simplicial complex $X$ we construct a chain homotopy inverse \[g \colon (\coCH_*(\mathcal{C}(\widetilde{X})))_{1\mathrm{-loc}} \to \Imm(\iota_X)\] for the inclusion $i: \Imm(\iota_X) \hookrightarrow (\coCH_*(\mathcal{C}(\widetilde{X})))_{1\mathrm{-loc}}$ by induction as follows. Choose a collection of contractible subsimplicial complexes $Z_A$ for all sub-simplicial complexes $A$ with $\diam A \leq 1$ as above. First define \[g_0 \colon ((\coCH_*(\mathcal{C}(\widetilde{X})))_{1\mathrm{-loc}})_0 \to \Imm(\iota_X)_0 \] by
    \[ g_0(v\{\sigma_1 | \dots |\sigma_N\}) := \iota_X(v)=v \cdot \text{id}_v,\] where $\sigma_1, \dots, \sigma_N \in X_1$ with $t(\sigma_i) =s(\sigma_{i+1})$ for $i=1,\dots,N-1$, and $s(\sigma_1)=v=t(\sigma_N) \in X_0.$  This map is well defined on $0$-th homology, for if 
    \[ x = \tau_0\{\tau_1|\dots|\tau_M\} \in ((\coCH_*(\mathcal{C}(\widetilde{X})))_{1\mathrm{-loc}})_1 \] 
    is a generator with $\partial x = w\{\sigma'_1 | \dots |\sigma'_{N'}\} - v\{\sigma_1 | \dots |\sigma_N\}$, we can let $g_1(x):= \iota_X(\tau_0)$ if $v \neq w$ (in which case $\tau_0 \in \widetilde{X}_1$) and $g_1(x):=0$ if $v=w$ (in which case $\tau_0 \in \widetilde{X}_0)$. Then $\partial g_1(x) = g_0(\partial x)$.
    
    Suppose that for a certain $k \ge 1$ we have defined $g_0,\dots, g_k$ such that $\partial g_i = g_{i-1} \partial$ for all $1 \le i \le k$, and $g_i(z)$ has support inside $Z_{\Supp(z)}$ for any generator $z$ of degree $i$. Let 
    \[ z \in ((\coCH_*(\mathcal{C}(\widetilde{X})))_{1\mathrm{-loc}})_{k+1} \]
    be a generator. Then by assumption $g_k(\partial z)$ is a closed element in $\Imm(\iota_X)_k$ with support inside $Z_{\Supp(\partial z)}$. Since $Z_{\Supp(\partial z)}$ is contractible, there is $g_{k+1}(z) \in \Imm(\iota_X)_{k+1}$  with support in $Z_{\Supp(\partial z)} \subset Z_{\Supp(z)}$ such that $\partial g_{k+1}(z) = g_k(\partial z)$. The construction of the chain homotopies $g \circ i \simeq \text{id}$ and $i \circ g \simeq \text{id}$ follows a similar argument. \end{proof}

As a consequence of \cref{prop:constantLoops} and \cref{prop:mLocalEquivalences}, we have the following result.
\begin{corollary}
    Suppose $K$ is an $m$-fine simplicial complex and $X$ as above. Then for any $m' \le m$ there is a quasi-isomorphism
    \[ \frac{\coCH_*(\calC(X))}{(\coCH_*(\calC(X)))_{m'\mathrm{-loc}}} \simeq  S_*(L|X|,|X|; R). \]
    between the quotient by the $m'$-local chains and the singular chains on the loop space relative to the constant loops.
\end{corollary}

\section{Local pairings and controlled topology}\label{sec:localpairings}
The setting for our formalism is a simplicial complex $K$ with a certain local pairing on chains that is supported sufficiently near the diagonal. We introduce a notion of locality and explain that the homology intersection pairing on a Poincar\'e duality complex $K$ can be lifted to a local pairing on chains if and only if $K$ is a homology manifold, following \cite{ranicki1999singularities}, \cite{ranicki1990chain}, \cite{mccrory77}, and \cite{mccrory1979zeeman}.

\subsection{Simplicial complexes with local pairing}
By a \textit{pairing} of degree $-n$ on a complex $V$ we shall mean a chain map $V \otimes V \to R$ of degree $-n$.
\begin{definition}
    A \emph{local pairing} of degree $-n$ on a simplicial complex $K$ is a pairing of degree $-n$
    \[ \theta \colon C_*(K) \otimes C_{*}(K) \to R \]
    such that $\theta(\sigma,\tau) = 0$ if $\overline\sigma \cap \overline\tau = \varnothing$, in other words, a pairing that vanishes on simplices whose closures do not overlap.
\end{definition}

Note that, via the Alexander-Whitney natural transformation $C_*(K \times K) \to C_*(K) \otimes C_*(K)$ any pairing of degree $-n$ gives rise to cochain in $C^n(K \times K)$ vanishing on pairs of simplices that do not intersect. 

Recall that $K$ is an $R$-homology $n$-\textit{Poincar\'e duality complex} if there is a fundamental chain $o_K \in C_n(K)$, so that the cap product
\[ [o_K]\frown - \colon H^{n-*}(K) \to H_*(K) \]
is an isomorphism. If $K$ is a Poincar\'e duality complex, so is the product $K \times K$ with fundamental chain given by $o_{K\times K}= EZ(o_K \otimes o_K) \in C_{2n}(K \times K)$, where $EZ \colon C_*(K) \otimes C_*(K) \to C_*(K \times K)$ denotes the Eilenberg-Zilber map. 

\begin{definition}\label{def:thetaNondeg}
    A pairing $\theta \colon C_*(K) \otimes C_*(K) \to R$ on an $R$-homology $n$-Poincar\'e duality complex $(K,o_K)$ is \emph{nondegenerate} if the class
    \[ [\theta] \in H^n(K \times K) \]
    is Poincar\'e dual in $K \times K$ to the diagonal fundamental class $\mathrm{diag}_*([o_K])$, that is, the pushforward of $[o_K]$ along the diagonal map $\mathrm{diag} \colon K \to K \times K$.
\end{definition}

\subsection{Controlled Poincar\'e duality and homology manifolds}
We now explain how the results of \cite{mccrory77} and 
\cite{ranicki1999singularities} imply that, for a Poincar\'e duality complex $K$, a local non-degenerate pairing exists if and only if $K$ is a homology manifold. 
Thus, homology manifolds provide a vast source of examples of simplicial complexes with local pairings. We start by recalling the notion of $(R,K)$-modules, also called `$K$-based' or `$K$-controlled' complexes.
\begin{definition}
    The category of $(R,K)$-modules has as objects finitely-generated free modules $M$ endowed with a direct sum decomposition
    \[ M = \bigoplus_{\sigma \in K} M_\sigma \]
    and as morphisms linear maps $M \to N$ satisfying
    \[ f(M_\sigma) \subset \bigoplus_{\tau \ge \sigma} N_\tau. \]
\end{definition}
In other words, the maps are usual linear maps but with the restriction that the elements supported on $\sigma$ can `spread to bigger simplices' immediately next to it, that is, have image supported in simplices $\tau$ whose closure contains $\sigma$.

Using the definition above, one straightforwardly defines (co)chain complexes of $(R,K)$-modules, morphisms and homotopies etc. by demanding all maps to be $(R,K)$-module morphisms.
\begin{remark}
    The simplicial \emph{cochain} complex $C^*(K)$ is naturally a complex of $(R,K)$-modules, since the cochain differential sends cochains supported on $\sigma$ to cochains supported on simplices $\tau$ whose boundary contains $\sigma$, that is, satisfying $\tau \ge \sigma$. On the other hand, the simplicial \emph{chain} complex together with the obvious direct sum decomposition is not, since the maps go in the wrong direction.
\end{remark}

Consider the barycentric subdivision $K'$ of $K$. By definition, its $k$-simplices are labeled by strictly increasing sequences
\[ (\sigma_0 < \sigma_1 < \dots < \sigma_k) \]
of simplices of $K$. For any $\sigma \in K$, let us denote by
\[ D(\sigma) = \{(\sigma_0 < \dots < \sigma_k)\ |\ \sigma \le \sigma_0 \} \]
the \emph{dual cell} of $\sigma$, which is the union of the cells labeled by a sequence starting with $\sigma$. The following fact just follows from the form of the differential on chains on $K'$.
\begin{proposition}
    The chain complex $M=C_*(K')$, endowed with the decomposition $M_\sigma = C_*(D(\sigma))$, is a chain complex of $(R,K)$-modules.
\end{proposition}

Let $o_K \in C_n(K)$ be any simplicial chain. We recall the `Flexner cap product' \cite{flexner1940simplicial}, which is a chain-level alternative to the more commonly known Whitney cap product, see \cite{mccrory1979zeeman} and \cite[Example~4.13]{ranicki1992algebraic}. Given any two simplices $\sigma$ and $\tau$ in $K$, we write $[\sigma:\tau]$ for the \emph{incidence number}, that is, the coefficient $\pm 1$ of $\sigma$ in $\delta\tau$ if $\sigma < \tau$ of codimension one, and zero otherwise. The Flexner cap product is then the map
\[ (- \underset{\mathrm{F}}{\frown} -) \colon C_*(K) \otimes C^*(K) \to C^*(K') \]
given on generators by
\[ \sigma \underset{\mathrm{F}}{\frown} \tau^\vee = \sum_{\{\sigma \le \sigma_0 < \dots < \sigma_k \le \tau\}} [\sigma_0:\sigma_1]\dots[\sigma_{k-1}:\sigma_k] (\sigma_0 < \dots < \sigma_k) \]
Note that the terms with nonzero coefficient in this sum must have $k = \dim(\tau) - \dim(\sigma)$. Note that by definition, for any chain $c \in C_k(K)$,
\[ - \underset{\mathrm{F}}{\frown} c \colon C^{k-*}(K) \to C_*(K') \]
is a map of complexes of $(R,K)$-modules.

To relate this cap product to the more conventional Whitney cap product, we can follow the latter by subdivision, which also gives a map
\[ C^*(K) \xrightarrow{- \frown c} C_{k-*}(K) \xrightarrow{\mathrm{sd}} C_{n-*}(K') \]
Note that $\mathrm{sd} \circ (- \frown c)$ is not necessarily a map of complexes of $(R,K)$-modules, since it does not satisfy the required locality condition. Nevertheless, these operations are equivalent by a local homotopy.
\begin{proposition}\label{prop:1localHomotopy}
    For any chain $c \in C_k(K)$, the maps $\underset{\mathrm{F}}{\frown} c$ and $\mathrm{sd} \circ (- \frown c)$ are 1-local homotopic, that is, related by a homotopy
    \[ h \colon C^{k-*}(K) \to C_{*+1}(K) \]
    such that $h(\tau^\vee)$ is supported on simplices $\sigma$ such that $\sigma \cap \tau \neq \varnothing$.
\end{proposition}
\begin{proof}
    Follows from the same proof as in \cite[Section 5]{mccrory1979zeeman}, keeping track of locality.
\end{proof}

The following definition is due to Ranicki.
\begin{definition}\label{def:KcontrolledFundChain}
    The chain $o_K$ is a \emph{$K$-controlled fundamental chain} if the map 
    \[ - \underset{\mathrm{F}}{\frown} o_K \colon C^{n-*}(K) \to C_{*}(K') \] 
    is a chain equivalence of $(R,K)$-modules.
\end{definition}

In other words, we ask that the inverse maps and homotopies which compose the data of a chain equivalence be maps of $(R,K)$-modules, instead of just requiring them to be graded linear. We now paraphrase \cite[Proposition~6.11]{ranicki1999singularities}.
\begin{theorem}\label{thm:KControlledFundChain}
    Let $(K,[o_K])$ be any $R$-homology $n$-Poincar\'e duality complex. Then $o_K$ is a $K$-controlled fundamental chain if and only if the class $U \in H^n(K\times K)$ that is Poincar\'e dual to $\Delta_*([o_K])$ has zero image in $H^n((K\times K) - \Delta(K))$.

    If these equivalent conditions hold, then the geometric realization $|K|$ is an \emph{$R$-homology manifold}, in the sense that for any point $x \in |K|$ we have
    \[ H_*(|K|,|K|- x;R) \cong \begin{cases}
        R &\text{if}\ *=n,\\
        0 &\text{otherwise,}
    \end{cases}\]
    in other words, its link is an $R$-homology $(n-1)$-sphere. 
\end{theorem}

We now relate this notion to our earlier notion of local pairing.
\begin{proposition}\label{prop:localPairingIffFundamentalChain}
    Let $(K,[o_K])$ be a $R$-homology $n$-Poincar\'e duality complex, such that $K$ is 1-fine. Then $K$ has a nondegenerate local pairing in the sense of \cref{def:thetaNondeg} if and only if $o_K$ is an $K$-controlled fundamental chain.
\end{proposition}
\begin{proof}
    For the `if' direction, when $o_K$ is an $K$-controlled fundamental chain, by definition there is a map of complexes
    \[ f \colon C_*(K') \to C^{n-*}(K) \]
    and a map
    \[ h \colon C_*(K') \to C_{*+1}(K') \]
    which is a homotopy between $f \circ (- \underset{\mathrm{F}}{\frown} o_K)$ and the identity. The fact that $f$ is a map $(R,K)$-modules implies that the element
    \[ \theta \colon C_*(K') \otimes C_{n-*}(K) \to R \]
    that is adjoint to it is a local pairing, and the fact that $h$ is a map of graded $(R,K)$-modules implies that $\theta$ is nondegenerate and that it is a 1-local homotopy.

    For the `only if' direction, suppose that we have a nondegenerate local pairing $\theta$. Since $K$ is 1-fine,
    \[ N_K = \{ (\sigma,\tau)\ |\ \overline\sigma \cap \overline\tau \neq \varnothing\} \]
    seen as a subset of $K\times K$, is a neighborhood for the diagonal, and the restriction map
    \[ H^*((K \times K)- \Delta(K)) \to H^*((K \times K) - N_K) \]
    is an isomorphism. So we pick $\theta'$ homologous to $\theta$ in the former group, and by \cref{thm:KControlledFundChain} we conclude that $o_K$ is a $K$-controlled fundamental chain.
\end{proof}

\section{Homotopy pairings on categorical coalgebras} \label{sec:homotopypairings}
In this section, we describe the notion of a homotopy pairing on a categorical coalgebra, which consists of a pairing together with higher homomotopies controlling compatibilities with the underlying coassociative coproduct. We work over an arbitrary commutative ring $R$. Similar constructions and results are obtained over the rationals in \cite{TZPD}. Let $C$ be a categorical coalgebra and consider the graded module
\begin{align*} 
    \Hom^*_{C_0^e\otimes C_0^e}(C\otimes C, \Omega C \otimes \Omega C) &= \prod_{a,b,c,d \in \mathcal{S}_C}\Hom_R^*(C(a,b)\otimes C(c,d),\Omega C (a,d) \otimes \Omega C(c,b)).
\end{align*}
For any $\varphi \in \Hom^*_{C_0^e\otimes C_0^e}(C\otimes C, \Omega C \otimes \Omega C)$ and $x \otimes y \in C \otimes C$ we write
\[\varphi(x,y)= \sum_{\beta} \varphi(x,y)'_\beta \otimes \varphi(x,y)''_\beta= \varphi(x,y)' \otimes \varphi(x,y)'' \in \Omega C \otimes \Omega C .\]
 We equip $\Hom^*_{C_0^e\otimes C_0^e}(C\otimes C, \Omega C \otimes \Omega C)$ with the differential given by
\begin{align}\label{eq:diffHom}
    (d\varphi)(x,y) = & d_{\Omega C \otimes \Omega C}(\varphi(x,y)) - (-1)^{\varphi}\ \varphi(d_{C \otimes C}(x,y)) \nonumber \\
    &-(-1)^{\varphi x^{(1)}}\ \{x^{(1)}\} \varphi(x^{(2)},y)' \otimes \varphi(x^{(2)},y)'' \nonumber \\
    &+(-1)^{x^{(2)}y + x^{(1)} + y + \varphi} \varphi(x^{(1)},y)' \otimes \varphi(x^{(1)},y)'' \{x^{(2)}\} \\
    &+(-1)^{\varphi(x,y^{(1)})'' y^{(2)} + \varphi(x,y^{(1)})'} \varphi(x,y^{(1)})'\{y^{(2)}\} \otimes \varphi(x,y^{(1)})'' \nonumber \\
    &-(-1)^{y^{(1)}y^{(2)} + \varphi(x,y^{(2)})'' y^{(1)} + \varphi(x,y^{(2)})'} \varphi(x,y^{(2)})' \otimes \{y^{(1)}\}\varphi(x,y^{(2)}) \nonumber
\end{align}
where $d_{C\otimes C}=d_C \otimes \text{id}_C + \text{id}_C \otimes d_C$, $d_{\Omega C \otimes \Omega C}$ is defined similarly, and the composition rule in $\Omega C$ is just given by juxtaposition. We may graphically represent this differential by
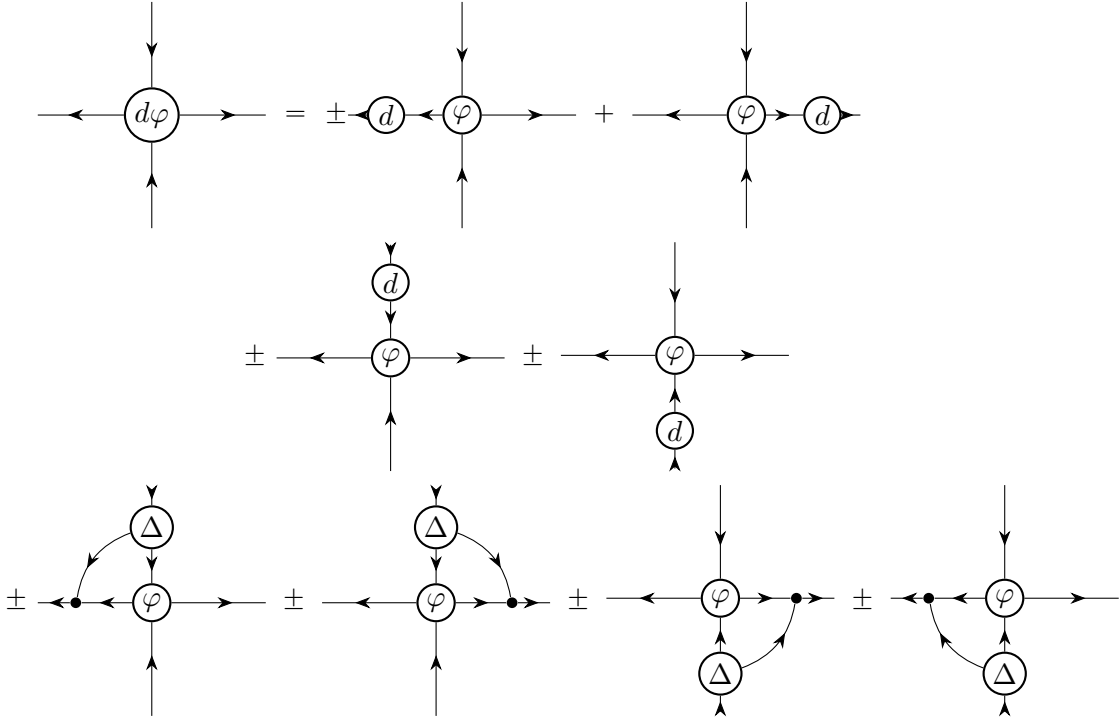
\begin{align*}
    \tikzfig{
        \node [vertex] (c) at (0,0) {$d\varphi$};
        \draw [->-] (0,1.5) to (c);
        \draw [->-] (0,-1.5) to (c);
        \draw [->-] (c) to (-1.5,0);
        \draw [->-] (c) to (1.5,0);
    }\ &=\ \pm\tikzfig{
        \node [vertex] (c) at (0,0) {$\varphi$};
        \node [vertex] (del) at (-1,0) {$d$};
        \draw [->-] (0,1.5) to (c);
        \draw [->-] (0,-1.5) to (c);
        \draw [->-] (c) to (del);
        \draw [->-] (del) to (-1.5,0);
        \draw [->-] (c) to (1.5,0);
    }\ + \ \tikzfig{
        \node [vertex] (c) at (0,0) {$\varphi$};
        \node [vertex] (del) at (1,0) {$d$};
        \draw [->-] (0,1.5) to (c);
        \draw [->-] (0,-1.5) to (c);
        \draw [->-] (c) to (del);
        \draw [->-] (del) to (1.5,0);
        \draw [->-] (c) to (-1.5,0);
    } \\
    \pm\ &\tikzfig{
        \node [vertex] (c) at (0,0) {$\varphi$};
        \node [vertex] (del) at (0,1) {$d$};
        \draw [->-] (0,1.5) to (del);
        \draw [->-] (del) to (c);
        \draw [->-] (0,-1.5) to (c);
        \draw [->-] (c) to (-1.5,0);
        \draw [->-] (c) to (1.5,0);
    }\ \pm\ \tikzfig{
        \node [vertex] (c) at (0,0) {$\varphi$};
        \node [vertex] (del) at (0,-1) {$d$};
        \draw [->-] (0,1.5) to (c);
        \draw [->-] (0,-1.5) to (del);
        \draw [->-] (del) to (c);
        \draw [->-] (c) to (-1.5,0);
        \draw [->-] (c) to (1.5,0);
    } \\
    \pm\ \tikzfig{
        \node [vertex] (c) at (0,0) {$\varphi$};
        \node [vertex] (del) at (0,1) {$\Delta$};
        \node [bullet] (m) at (-1,0) {};
        \draw [->-] (0,1.5) to (del);
        \draw [->-] (del) to (c);
        \draw [->-] (0,-1.5) to (c);
        \draw [->-] (c) to (m);
        \draw [->-] (m) to (-1.5,0);
        \draw [->-] (c) to (1.5,0);
        \draw [->-, bend right] (del) to (m);
    }\ &\pm \ \tikzfig{
        \node [vertex] (c) at (0,0) {$\varphi$};
        \node [vertex] (del) at (0,1) {$\Delta$};
        \node [bullet] (m) at (1,0) {};
        \draw [->-] (0,1.5) to (del);
        \draw [->-] (del) to (c);
        \draw [->-] (0,-1.5) to (c);
        \draw [->-] (c) to (-1.5,0);
        \draw [->-] (c) to (m);
        \draw [->-] (m) to (1.5,0);
        \draw [->-, bend left] (del) to (m);
    }\ \pm \ \tikzfig{
        \node [vertex] (c) at (0,0) {$\varphi$};
        \node [vertex] (del) at (0,-1) {$\Delta$};
        \node [bullet] (m) at (1,0) {};
        \draw [->-] (0,1.5) to(c);
        \draw [->-] (0,-1.5) to (del);
        \draw [->-] (del) to (c);
        \draw [->-] (c) to (-1.5,0);
        \draw [->-] (c) to (m);
        \draw [->-] (m) to (1.5,0);
        \draw [->-, bend right] (del) to (m);
    }\ \pm \ \tikzfig{
        \node [vertex] (c) at (0,0) {$\varphi$};
        \node [vertex] (del) at (0,-1) {$\Delta$};
        \node [bullet] (m) at (-1,0) {};
        \draw [->-] (0,1.5) to(c);
        \draw [->-] (0,-1.5) to (del);
        \draw [->-] (del) to (c);
        \draw [->-] (c) to (m);
        \draw [->-] (m) to (-1.5,0);
        \draw [->-] (c) to (1.5,0);
        \draw [->-, bend left] (del) to (m);
    }
\end{align*}
where the $C$ inputs are top and bottom, and the $\Omega C$ outputs and left and right, in those orders, and we omit the signs.

\begin{definition}
    A \textit{homotopy pairing} of degree $-n$ on a categorical coalgebra $C$ is an element \[\alpha \in \Hom^n_{C_0^e\otimes C_0^e}(C\otimes C, \Omega C \otimes \Omega C)\] such that $d \alpha=0$.
\end{definition}

\begin{remark}
    A homotopy pairing is a coalgebraic version of a $2$-truncated pre-Calabi Yau algebra in the language of \cite{KTV} or of a $V_2$-algebra, that is, an $A_{\infty}$-algebra with invariant and symmetric homotopy co-inner product, in the language of \cite{TZ}. We note that there is no non-degeneracy condition in the definition of a homotopy pairing. The signs in \cref{eq:diffHom} can be obtained by translating the signs for pre-Calabi--Yau elements to the context of categorical coalgebras used here.
\end{remark}

\begin{example} 
    Suppose $A$ is a finite-dimensional dg connected associative algebra over a field $\mathbb{k}$ equipped with a non-degenerate degree $-n$ pairing $\langle- , -\rangle \colon A \otimes A \to \mathbb{k}$. By non-degeneracy, the induced map $\rho \colon A \to \mathrm{Hom}_\mathbb{k}(A,\mathbb{k})$, given by $\rho(a)=\langle -,a\rangle$, is an isomorphism of degree $-n$. Furthermore, let us suppose that the following two compatibility equations hold
    \begin{align} \label{eq:twocompatibilities}
        \langle ab,c\rangle= \langle a, bc\rangle \text{ and } \langle a,bc\rangle= (-1)^{c(a+b)} \langle ca,b\rangle.
    \end{align}
    When $A$ is unital, the equations above imply that the pairing is symmetric. Since $A$ is finite dimensional as a $\mathbb{k}$-vector space, the linear dual of the product of $A$ makes $C=\mathrm{Hom}_\mathbb{k}(A,\mathbb{k})$ into a dg connected coassociative coalgebra, and thus $\Omega C$ is a dg associative algebra. The map
    \[ \alpha \colon C \otimes C \xrightarrow{\rho^{-1} \otimes \rho^{-1}} A \otimes A \xrightarrow{\langle- ,- \rangle} \mathbb{k} \cong \mathbb{k} \otimes \mathbb{k} \hookrightarrow \Omega C \otimes \Omega C\]
    is an example of a homotopy pairing on $C$, since the compatibility relations in \cref{eq:twocompatibilities} imply that $d \alpha=0$. This example  includes symmetric Frobenius algebras and also the Poincar\'e duality $\mathbb{Q}$-algebra models of \cite{LS} for simply connected closed oriented manifolds. 
\end{example}

We now fix a simplicial set $X$ obtained by a coherent ordering of the vertices in each simplex of an underlying $1$-fine finite connected simplicial complex. Let $\sigma$ and $\tau$ be two simplices of $X$ such that $\overline\sigma \cap \overline\tau \neq \varnothing$, and let $v,w$ be any pair of vertices of $\sigma, \tau$, respectively. There is a distinguished class
\[ [p_{v,w}] \in H_0(\Omega\mathcal{C}(\widetilde{X}) (v,w)) \]
which can be represented by any simplicial path between $v$ and $w$ in $\sigma \cup \tau$. This class is uniquely defined and depends only on $v$ and $w$, and not on $\sigma$ and $\tau$, because of the assumption that $X$ is 1-fine. 

We now show that any local pairing $\theta$ on the underlying simplicial complex can be lifted to a local homotopy pairing. We first lift the pairing to $\calC(\widetilde{X})$ by setting
\[ \widetilde\theta(\check{\sigma},z) = - \theta(\sigma,z), \quad \widetilde\theta(\check{\sigma},\check{\tau}) =  \theta(\sigma,\tau), \quad \widetilde\theta(x^{k}_\sigma,z) = \widetilde\theta(y^{k}_\sigma,z) = 0\]
for any $\sigma,\tau \in X_1$, $k \ge 2$ and $z \in \calC(\widetilde{X})$.
\begin{theorem}\label{thm:alphaExistence}
    Given any local pairing $\theta$ on the underlying simplicial complex of $X$, there is a homotopy pairing
    \[ \alpha \colon \mathcal{C}(\widetilde{X}) \otimes \mathcal{C}(\widetilde{X}) \to \Omega \mathcal{C}(\widetilde{X}) \otimes \Omega \mathcal{C}(\widetilde{X}) \]
    such that following conditions hold for simplices $\sigma$ and $\tau$ in $\widetilde{X}$:
    \begin{enumerate}
        \item if $\sigma \cap \tau =\emptyset$, then $\alpha(\sigma,\tau) = 0$,
        \item decomposing $\alpha(\sigma,\tau) = \sum_{a \in I} \alpha(\sigma,\tau)'_a \otimes \alpha(\sigma,\tau)''_a$ as a sum of simple tensors of generators, we have
        \[ \Supp(\alpha(\sigma,\tau)'_a) \cup \Supp(\alpha(\sigma,\tau)''_a) \subset \sigma \cup \tau, \]
        and 
        \item if $\dim(\sigma) + \dim(\tau) = n$, then
     \[ [\alpha(\sigma,\tau)] = \widetilde{\theta}(\sigma,\tau) \cdot [p_{s(\sigma),t(\tau)} \otimes p_{s(\tau),t(\sigma)}].\]
    \end{enumerate}
\end{theorem}

\begin{proof}
    We produce $\alpha$ using an inductive argument on the sum of dimensions of the inputs. This argument has two special cases in the beginning, at degrees $n$ and $n+1$, and then continues inductively. Note that we must have $\alpha(\sigma,\tau)=0$ if that sum is smaller than $n$, by degree reasons. So we start the process with all pairs $\sigma,\tau$ such that $\dim(\sigma) + \dim(\tau) = n$. If $\sigma \cap \tau = \varnothing$, we set $\alpha(\sigma,\tau) = 0$, and otherwise, we pick any paths
    \[ p_{s(\sigma),t(\tau)} \in \Omega \mathcal{C}(\widetilde{X})(s(\sigma),t(\tau)), \quad p_{s(\tau),t(\sigma)}  \in \Omega \mathcal{C}(\widetilde{X})(s(\tau),t(\sigma)) \]
    in $\sigma \cup \tau$ and set
    \[ \alpha(\sigma,\tau) = \widetilde{\theta}(\sigma,\tau) \cdot \left(p_{s(\sigma),t(\tau)} \otimes  p_{s(\tau),t(\sigma)}\right) \in \Omega \mathcal{C}(\widetilde{X})(s(\sigma),t(\tau)) \otimes \Omega \mathcal{C}(\widetilde{X})(s(\tau),t(\sigma)).  \]

    Going one dimension up, given any pair of simplices with $\dim(\sigma) + \dim(\tau) = n+1$, in order for $\alpha$ to be closed we need to find a value for $\alpha(\sigma,\tau)$ satisfying
    \begin{align}\label{eq:alphaClosed}
        d_{\Omega \mathcal{C}(\widetilde{X})\otimes \Omega \mathcal{C}(\widetilde{X})}(\alpha(\sigma, \tau)) =\ &+(-1)^{n}\ \alpha(d_{C \otimes C}(\sigma,\tau)) \nonumber \\
        &+(-1)^{n \sigma^{(1)}}\ \{\sigma^{(1)}\} \alpha(\sigma^{(2)},\tau)' \otimes \alpha(\sigma^{(2)},\tau)'' \nonumber \\
        &+(-1)^{\sigma^{(2)}\tau + \sigma^{(1)} + \tau + n + 1} \alpha(\sigma^{(1)},\tau)' \otimes \alpha(\sigma^{(1)},\tau)'' \{\sigma^{(2)}\} \\
        &+(-1)^{\alpha(\sigma,\tau^{(1)})'' \tau^{(2)} + \alpha(\sigma,\tau^{(1)})'+1} \alpha(\sigma,\tau^{(1)})'\{\tau^{(2)}\} \otimes \alpha(\sigma,\tau^{(1)})'' \nonumber \\
        &+(-1)^{\tau^{(1)}\tau^{(2)} + \alpha(\sigma,\tau^{(2)})'' \tau^{(1)} + \alpha(\sigma,\tau^{(2)})'} \alpha(\sigma,\tau^{(2)})' \otimes \{\tau^{(1)}\}\alpha(\sigma,\tau^{(2)}) \nonumber
    \end{align}
    It suffices to show that the homology class of the right-hand side is zero in $H_0(\Omega \mathcal{C}(\widetilde{X}) \otimes \Omega \mathcal{C}(\widetilde{X}))$. For that, let us write it explicitly, for $\sigma = (v_0,\dots,v_p), \tau = (w_0,\dots,w_q)$, with $p+q=n+1$. We note that there are $(p-1)+(q-1)$ terms coming from the internal differential, and only four nontrivial terms involving the coproduct. Multiplying by an overall sign we have
    \begin{align*}
        (-1)^{n} &d_{\Omega \mathcal{C}(\widetilde{X})\otimes \Omega \mathcal{C}(\widetilde{X})}(\alpha(\sigma, \tau)) = +\sum_{i=1}^{p-1} (-1)^i \alpha(v_0,\dots,\widehat{v_i},\dots,v_p, w_0,\dots,w_q) \\
        &+\sum_{j=1}^{p-1} (-1)^{p+j} \alpha(v_0,\dots,v_p, w_0,\dots,\widehat{w_j},\dots,w_q) \\
        &+\{v_0,v_1\}\alpha(v_1,\dots,v_p, w_0,\dots,w_q)'\otimes \alpha(v_1,\dots,v_p, w_0,\dots,w_q)'' \\
        &+(-1)^{p}\ \alpha(v_0,\dots,v_{p-1}, w_0,\dots,w_q)' \otimes \alpha(v_0,\dots,v_{p-1}, w_0,\dots,w_q)''\{v_{p-1},v_p\}\\
        &+(-1)^{p+q}\ \alpha(v_0,\dots,v_p, w_0,\dots,w_{q-1})'\{w_{q-1},w_q\} \otimes \alpha(v_0,\dots,v_p, w_0,\dots,w_{q-1})''\\
        &+(-1)^p\ \alpha(v_0,\dots,v_p, w_1,\dots,w_q)'\otimes \{w_0,w_1\} \alpha(v_0,\dots,v_p, w_1,\dots,w_q)''
    \end{align*}
    We see that each term in the right-hand side represents a multiple of the same class 
    \[ [p_{s(\sigma),t(\tau)}] \otimes [p_{s(\tau),t(\sigma)}]. \] 
    They correspond, with correct signs, to the terms appearing in the expression for
    \[ \widetilde\theta \circ (\delta \otimes \id + \id \otimes \delta) (\sigma,\tau), \]
    with $\delta$ being the usual differential on normalized chains (and not the modified differential). Since $\widetilde{\theta}$ is compatible with the differential, the right-hand side represents the zeroth homology class and is exact.

    We then proceed inductively on the total dimension. For $\dim(\sigma) + \dim(\tau) = N \ge n+2$, the right-hand side of \cref{eq:alphaClosed} is a closed element of degree $N-n-1$ and represents a homology class of that degree in the path space of $\sigma \cup \tau$. By 1-fineness, this space is contractible so we pick $\alpha(\sigma,\tau)$ to be any primitive with support contained in $\sigma \cup \tau$.
\end{proof}

 Any homotopy pairing on $\mathcal{C}(\widetilde{X})$ satisfying (1) and (2) in \cref{thm:alphaExistence} will be called a \textit{local homotopy pairing}. If it satfies (3) we will say it \textit{lifts} the pairing $\theta$. When constructing $\alpha$ above, choices were made at each inductive step. Nevertheless, the cohomology class of $\alpha$ only depends on that of $\theta$.
\begin{proposition}\label{prop:alphaUnique}
    The cohomology class \[[\alpha] \in H^n\big(\Hom^*_{\mathcal{C}(\widetilde{X})_0^e\otimes \mathcal{C}(\widetilde{X})_0^e}(\mathcal{C}(\widetilde{X})\otimes \mathcal{C}(\widetilde{X}), \Omega \mathcal{C}(\widetilde{X}) \otimes \Omega \mathcal{C}(\widetilde{X})),d\big)\] is independent of the choices made in the proof of \cref{thm:alphaExistence} and only depends on the cohomology class $[\theta] \in H^n(X \times X)$.
\end{proposition}
\begin{proof}
    Let $\theta_1,\theta_2$ be two local pairings representing same cohomology class, and $\alpha_1, \alpha_2$ obtained by applying the inductive argument to them, respectively, using any intermediate choices. We can use the same inductive argument (over the total input degree) to construct a homotopy between $\alpha_1,\alpha_2$, that is, a map of graded modules
    \[ \beta \colon \mathcal{C}(\widetilde{X}) \otimes \mathcal{C}(\widetilde{X}) \to\Omega \mathcal{C}(\widetilde{X}) \otimes \Omega \mathcal{C}(\widetilde{X}) \]
    of degree $n-1$ such that $d\beta = \alpha_1 - \alpha_2$.
\end{proof}

\section{Modeling the loop product and coproduct}

In this section we give explicit formulae for the algebraic models of the loop product and coproduct operations. Suppose that $C$ is any categorical coalgebra and \[\alpha \in \Hom^n_{C_0^e \otimes C_0^e}(C \otimes C, \Omega C\otimes \Omega C), \]
a homotopy pairing. The example to have in mind is $C= \mathcal{C}(\widetilde{X})$ and $\alpha$ arising from an intersection pairing. For simplicity, we denote the generators of the coHochschild complex $\coCH_*(C)$ by $(c,a)$ where $c \in C$ and $a \in \Omega C$.

\subsection{The algebraic loop product}

\begin{definition}\label{def:algebraicLoopProduct}
    The \emph{algebraic loop product} $\mu_{\alpha}$ associated to $(C,\alpha)$ is the map of degree $-n$ 
    \[ \mu_{\alpha} \colon \coCH_*(C) \otimes \coCH_*(C) \to \coCH_*(C) \]
    given by
    \[ \mu_\alpha \big( ( c_1, a_1) \otimes (c_2, a_2)\big) = (-1)^s \big(c_1^{(2)}, a_1 \cdot \alpha(c_1^{(1)},c_2)'\cdot a_2 \cdot \alpha(c_1^{(1)},c_2)''\big) \]
    where the sign is given $(-1)^n$ times the Koszul sign taking into account the degree of the operation $\alpha$, explicitly
    \[ s = n + c_1^{(1)}c_1^{(2)} + c_1^{(1)}a_1 + n c_1^{(2)} + n a_1 + a_2\  \alpha(c_1^{(1)},c_2)''. \]
\end{definition}

\noindent 

We may express the algebraic loop product by the following diagram

\[\tikzfig{
    \node [circ] (in1) at (-2,0) {};
    \node [cross] at (-2,0) {};
    \node [circ] (in2) at (0,0) {};
    \node [cross] at (0,0) {};
    \node [bullet] (w) at (-1,0) {};
    \node [vertex] (alpha) at (0,1) {$\alpha$};
    \node [bullet] (e) at (1,0) {};
    \node [bullet] (s) at (0,-1) {};
    \node [vertex,draw=red,red] (Delta) at (-3,0) {$\Delta$};
    \node [circ] (out) at (0,-2) {};
    \draw [->-,red] (in1) -- (Delta);
    \draw [->-,rounded corners,red] (Delta) -- (-3,2) -- (0,2) -- (alpha);
    \draw [->-] (in1) -- (w);
    \draw [->-,red] (in2) -- (alpha);
    \draw [->-,rounded corners] (alpha) -- (-1,1) -- (w);
    \draw [->-,rounded corners] (alpha) -- (1,1) -- (e);
    \draw [->-,rounded corners] (w) -- (-1,-1) -- (s);
    \draw [->-,rounded corners] (e) -- (1,-1) -- (s);
    \draw [->-] (in2) -- (e);
    \draw [->-] (s) -- (out);
    \draw [->-,red] (Delta) -- (-3,-3) -- (0,-3) -- (out);
}\]
In the diagram above, the red lines carry $C$ factors, the black lines carry $\Omega C$ factors, the dots are composition in $\Omega C$, the crossed circles are inputs, and the empty circles are outputs. We choose the left crossed circle to be the first input.

\begin{proposition} \label{prop:product}
    For any homotopy pairing $\alpha$ on $C$ the algebraic loop product $\mu_{\alpha}$ is a chain map. In particular, any homotopy pairing of degree $-n$ on $\mathcal{C}(\widetilde{X})$ gives a product of degree $-n$ on $H_*(L|X|)$.
\end{proposition}
\begin{proof}
    Follows immediately from the definition of the differential of $\Hom^*_{C_0^e \otimes C_0^e}(C \otimes C, \Omega C \otimes \Omega C)$, and from the fact that the differential on $\Omega C$ is a derivation for its composition map.
\end{proof}

\subsection{The algebraic loop coproduct}
In order to define our algebraic model for the loop coproduct we start by introducing a degree $+1$ algebraic analogue of the $1$-parameter family of maps that ``split" or ``cut" a path at all possible times. 

\subsubsection{The scanning map}
Consider the dg $R$-module $\Omega C= \bigoplus_{v,w \in \mathcal{S}_C} \Omega C(v,w)$ equipped with the $C_0$-bicomodule structure given by the source and target maps. Consider the graded $R$-module
\[\Omega C \boxtimes C \boxtimes \Omega C, \] 
where the cotensor product $\boxtimes$ is taken over $C_0$, endowed with the differential
\begin{align*} \label{markedpaths}
    \partial(a \otimes c \otimes b) &= d_{\Omega C} a \otimes c \otimes b + (-1)^{a} a \otimes d_Cc \otimes b + (-1)^{a + c + 1} a \otimes c \otimes d_{\Omega C} b 
    \\
    & (-1)^{a} a \{c^{(1)}\} \otimes c^{(2)} \otimes b + (-1)^{|a|+|c^{(1)}|+1} a \otimes c^{(1)} \otimes \{c^{(2)}\}b.
\end{align*}
In the case $C=\mathcal{C}(\widetilde{X})$, we may use the maps from $\mathcal{C}(\widetilde{X})$ and $\Omega \mathcal{C}(\widetilde{X})$ to singular chains on path spaces, for any $v,w \in X_0$ we get a quasi-isomorphism
\[ \leftindex_{v}(\Omega \mathcal{C}(\widetilde{X}) \boxtimes \mathcal{C}(\widetilde{X}) \boxtimes \Omega \mathcal{C}(\widetilde{X}))_w \xrightarrow{\sim} S_*(P^*_{v,w}|X|), \]
where the target denotes the the singular chains on the \emph{marked} path space, i.e., the space of paths from $v$ to $w$ in $X$ together with a distinguished point along the path, and the left hand side is equipped with the differential $\partial$ defined above. If $\sigma \in \mathcal{C}(\widetilde{X})$ this map sends $a\otimes \sigma \otimes b \in \leftindex_{v}(\Omega \mathcal{C}(\widetilde{X}) \boxtimes \mathcal{C}(\widetilde{X}) \boxtimes \Omega \mathcal{C}(\widetilde{X}))_w$ to a family of marked paths in $|X|$ from $v=s(a)$ to $w=t(b)$ of dimension $|a|+|\sigma|+|b|$ with marked points lying inside $\sigma \subset |X|$.

\begin{definition}
    The \emph{scanning map}
    \[ \cals \colon \Omega C \to \Omega C \boxtimes C \boxtimes \Omega C \]
    is the degree $+1$ map of graded modules given by
    \begin{align*} 
        \cals(\{c_1|\dots|c_N\}) &= \gamma(\rho_l(c_1)')\ \otimes\ \rho_l(c_1)''\ \otimes\ \{c_2|\dots|c_N\} \\
        &+ \sum_{i=2}^{N-1} (-1)^{c_1 + \dots + c_{i-1} + i-1} \{c_1|\dots|c_{i-1}\}\ \otimes\ c_i\ \otimes\ \{c_{i+1}|\dots|c_N\} \\
        &+ (-1)^{c_1 + \dots + c_N} \{c_1|\dots|c_{N-1}\}\ \otimes\ \rho_r(c_N)'\ \otimes\ \gamma(\rho_r(c_N)''),
    \end{align*}
    where as before $\rho_l \colon C \to C_0 \otimes C$ and $\rho_r \colon C \to C \otimes C_0$ are the $C_0$-bicomodule structure maps and $\gamma \colon C_0 \to \Omega C$ sends each $a \in \mathcal{S}(C)$ to its identity morphism.
\end{definition}

The scanning map is a chain homotopy between two chain maps determined by the $C_0$-bicomodule structure of $\Omega C$ given by the source and target maps as stated in the following proposition.
\begin{proposition}\label{prop:scanFailure}
    The scanning map fails to be a chain map by an error term given by
    \begin{align*} 
        d(\cals(\{c_1|\dots|c_N\})) + \cals(d(\{c_1|\dots|c_N\})) &= \gamma(\rho_l(\rho_l(c_1)')' \ \otimes\ \rho_l(\rho_l(c_1)')''\ \otimes\ \{\rho_l(c_1)''|\dots|c_N\} \\
        &\pm \{c_1|\dots|\rho_r(c_N)'\} \otimes \rho_r(\rho_r(c_n)'')'\ \otimes\ \gamma(\rho_r(\rho_r(c_n)'')''),
    \end{align*}
    for any $\{c_1|\dots|c_N\} \in \Omega C$.
\end{proposition}
\begin{proof} This is checked by direct computation. 
\end{proof}
When interpreted in terms of path spaces, the scanning map may be thought of a chain homotopy between two chain maps induced by the continuous maps
\[P|X| \to |X| \times P|X| \text{   and   } P|X| \to P|X| \times X\] 
given by $(\gamma,T)
\mapsto (\gamma(0), (\gamma,T))$
and
$(\gamma,T)\mapsto ((\gamma,T), \gamma(T))$, respectively. 

We write
\[ \cals(a) = \cals^{l} a \otimes \cals^c a \otimes \cals^r a \]
for the scanning map. Since the coproduct $\Delta$ of $C$ is coassociative, there is no ambiguity in using Sweedler's notation to write
\[ \Delta^2(c) \coloneqq (\Delta \otimes \id)(\Delta(c)) = (\id \otimes \Delta)(\Delta(c)) = c^{((1))} \otimes c^{((2))} \otimes c^{((3))}.\]

\begin{definition}\label{def:algloopcoproduct}
    The \emph{algebraic loop coproduct} associated to $(C,\alpha)$ is the map of degree $1-n$
    \[ \lambda_\alpha \colon \coCH_*(C) \to \coCH_*(C)\otimes \coCH_*(C)\]
    given by
    \[ \lambda_\alpha(c,a) = \pm \left(c^{((3))}, \cals^l a \  \alpha(c^{((2))},\cals^c a)'' \right) \otimes \left(c^{((1))}, \alpha(c^{((2))},\cals^c a)' \  \cals^r a\right)\]
    where the sign is given by $(-1)^{n+c}$ times the Koszul sign, taking into account the degrees of $\alpha$ and $\cals$ and starting from the order $\alpha,\cals,c,a$.
\end{definition}

We may express the algebraic loop coproduct by the following diagram

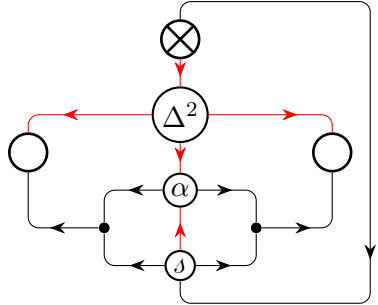
\[ \tikzfig{
    \node [vertex] (alpha) at (0,0) {$\alpha$};
    \node [vertex] (nabla) at (0,1) {$\Delta^2$};
    \node [circ] (in) at (0,2) {};
    \node [cross] at (0,2) {};
    \node [vertex] (scan) at (0,-1) {$\cals$};
    \node [bullet] (l) at (-1,-0.5) {};
    \node [bullet] (r) at (1,-0.5) {};
    \node [circ] (out1) at (-2,0.5) {};
    \node [circ] (out2) at (2,0.5) {};
    \draw [->-,red] (in) -- (nabla);
    \draw [->-,rounded corners,red] (nabla) -- (-2,1) -- (out1);
    \draw [->-,rounded corners,red] (nabla) -- (2,1) -- (out2);
    \draw [->-,red] (nabla) -- (alpha);
    \draw [->-,red] (scan) -- (alpha);
    \draw [->-,rounded corners] (in) -- (0,2.5) -- (2.5,2.5) -- (2.5,-1.5) -- (0,-1.5) -- (scan);
    \draw [->-=0.4,rounded corners] (scan) -- (-1,-1) -- (l);
    \draw [->-=0.4,rounded corners] (alpha) -- (-1,0) -- (l);
    \draw [->-=0.4,rounded corners] (scan) -- (1,-1) -- (r);
    \draw [->-=0.4,rounded corners] (alpha) -- (1,0) -- (r);
    \draw [->-=0.4,rounded corners] (l) -- (-2,-0.5) -- (out1);
    \draw [->-=0.4,rounded corners] (r) -- (2,-0.5) -- (out2);
}
\]
where we used the same conventions as in the diagram for the product, and the outputs are read with the \emph{right-hand side one first}. This operation does not give a chain map since $\cals$ fails to be a chain map, but we can control this failure.

\begin{proposition}\label{prop:coprodFailure}
    The algebraic loop coproduct satisfies the following equation.
    \begin{align*}
        (\partial \otimes 1 + 1 \otimes \partial) & \lambda_\alpha(c_0\{c_1|\dots|c_N\}) - (-1)^{n-1} \lambda_\alpha(\partial(c_0\{c_1|\dots|c_N\})) = \\
        &\pm \left(c_0^{((3))}, \alpha(c_0^{((2))},\rho_l(c_1)')'' \right) \otimes \left(c_0^{((1))}, \alpha(c_0^{((2))},\rho_l(c_1)')' \ \{\rho_l(c_1)''|\dots|c_N\} \right) \\
        &\pm \left(c_0^{((3))}, \{c_1|\dots|\rho_r(c_N)'\} \  \alpha(c_0^{((2))},\rho_r(c_N)'')'' \right) \otimes \left(c_0^{((1))}, \alpha(c_0^{((2))},\rho_r(c_N)''\right)
    \end{align*}
\end{proposition}
\begin{proof}
    Follows from a direct computation using \cref{prop:scanFailure}.
\end{proof}

\begin{proposition} \label{prop:coprodctmodconstant}
    Suppose $X$ is a simplicial set obtained from a $1$-fine simplicial complex and $\alpha$ is a local homotopy pairing on $\calC= \mathcal{C}(\widetilde{X})$. Then the algebraic loop coproduct associated to $(\calC,\alpha)$ induces  a chain map 
    \[ \lambda_\alpha \colon \frac{\coCH_*(\calC)}{(\coCH_*(\calC))_{1\mathrm{-loc}}} \to \frac{\coCH_*(\calC) \otimes \coCH_*(\calC)}{(\coCH_*(\calC))_{1\mathrm{-loc}} \otimes \coCH_*(\calC) + \coCH_*(\calC) \otimes (\coCH_*(\calC))_{1\mathrm{-loc}}} \] of degree $1-n$. Consequently, when $R$ is a field, $\lambda_{\alpha}$ defines a coproduct of degree $1-n$ on $H_*(L|X|,|X|)$.
\end{proposition}

\begin{proof}
    In other words, the proposition above says that $\lambda_\alpha$ becomes a chain map when composed to the canonical map to the quotient of the target, which then factors through the quotient of the source; we will prove these claims sequentially. For the first claim, we apply the formula of \cref{prop:coprodFailure} to the generator $(c,a) = \sigma_0\{\sigma_1|\dots|\sigma_N\}$. Observing the right-hand side, by locality of $\alpha$, we conclude that the first factor of the first term and the second factor of the second term have support contained in $\sigma_0$, so the right-hand side belongs to
    \[(\coCH_*(\calC))_{1\mathrm{-loc}} \otimes \coCH_*(\calC) + \coCH_*(\calC) \otimes (\coCH_*(\calC))_{1\mathrm{-loc}}. \]
    As for the second claim, if the support of $a$ is contained in some simplex $\sigma$, then, by locality of $\alpha$, so are both the left and right outputs, and all the terms in the right-hand side of the equation in \cref{prop:coprodFailure} are tensors of 1-local elements.
\end{proof}

\subsection{Comparison}
We will now compare the algebraic loop product and coproduct to the geometrically-defined Chas--Sullivan product and Goresky--Hingston coproduct, in the case where the simplicial complex $K$ comes from a triangulated closed oriented smooth manifold and its local pairing $\theta$ comes from a Thom cochain.

\subsubsection{Recap of definitions}\label{sec:recap}
Let $(M,g)$ be a smooth $n$-dimensional Riemannian manifold and $K$ a 2-fine simplicial complex (\cref{def:fine}) given by a triangulation of $M$. For ease of notation we will identify simplices of $K$ with their image as a subset of $M$. We consider the diagonal $M \subset M \times M$ and its neighborhoods
\[ N_\epsilon = \{(p,p')\in M \times M\ |\ \mathrm{dist}_g(p,p') \le \epsilon \} \]
for $\epsilon >0$ much smaller than the injectivity radius. The following is straightforward to check.
\begin{proposition}\label{prop:epsilonSimplices}
    For a small enough $\epsilon > 0$, the neighborhood $N_\epsilon$ of the diagonal is contained in the subset
    \[ \bigcup_{\sigma \cap \tau \neq \varnothing} \sigma \times \tau \]
    where $\sigma$ and $\tau$ are simplices of $K$.
\end{proposition}

The choice of auxiliary metric $g$ and parameter $\epsilon$ defines a geodesic `connecting' map
\[ \varpi_{g,\epsilon} \colon N_\epsilon \to PM \]
to the path space of $M$, by taking a pair of points that are  at most $\epsilon$-close to the unique geodesic of length $\le \epsilon$ connecting them. This construction is used to create intersections geometrically to define the loop product and coproduct in \cite[Section~2]{naef2023string}. We will use a modification of this map that will be compatible with our combinatorial models. The following statement follows from the piecewise-linear approximation of smooth paths and the regularity of the neighborhood of submanifolds, in this case, $M \subset M\times M$.
\begin{proposition}
    There is a homotopy between the map $\varpi_{g,\epsilon}$ and a map
    \[ \varpi_\mathrm{pl} \colon N_\epsilon \to PM \]
    such that, for any $(p,p') \in N_\epsilon$, the path $\varpi_\mathrm{pl}(x,y)$ is \emph{piecewise linear} for the cubical coordinates and contained in $\sigma \cup \sigma'$, where $\sigma$ and $\sigma'$ are any two simplices of $K$ respectively containing $p$ and $p'$ with $\sigma \cap \sigma' \neq \varnothing$.
\end{proposition}

We can thus replace the geodesic connecting map used in the definitions of the loop product and coproduct by the corresponding maps defined by $\varpi_\mathrm{pl}$ instead, obtaining chain-equivalent operations. Given any cochain
\[ \tau \in S^n(M \times M,( M \times M) - M) \simeq S^n(TM, TM-M) \]
representing the Thom class of $M$; we obtain a chain level representative for the Chas--Sullivan product of degree $-n$
\[ \mu_\mathrm{CS} \colon S_*(LM) \otimes S_*(LM) \to S_*(LM), \]
and the Goresky--Hingston coproduct of degree $(1-n)$
\[ \lambda_\mathrm{GH} \colon \frac{S_*(LM)}{S_*(M)} \to \frac{S_*(LM) \otimes S_*(LM)}{S_*(M) \otimes S_*(LM) + S_*(LM) \otimes S_*(M)}. \]

Let us describe in more detail how our versions of these operations are defined, appropriately modifying \cite[Definitions~2.1,~2.2]{naef2023string}. For the Chas--Sullivan product, one considers pullback squares
\[\xymatrix{
    \mathrm{Fig}(8) \ar@{^{(}->}[r] \ar[d] & V_\epsilon \ar[d] \ar@{^{(}->}[r] & LM \times LM \ar[d]^{\mathrm{ev}_0 \times \mathrm{ev}_0} \\
    M \ar@{^{(}->}[r] & N_\epsilon \ar@{^{(}->}[r] & M \times M
}\]
Note that the locus over the diagonal $M$ consists of pairs of loops with the same base-point, hence the name `figure eight'. We write
\[ \theta_\mathrm{CS} = (\mathrm{ev}_0 \times \mathrm{ev}_0)^*\tau \in S^n(LM \times LM, (LM \times LM) -\mathrm{Fig}(8))\]
for the pullback of the Thom cochain. We have a natural map
\[ S_*(V_\epsilon, V_\epsilon -\mathrm{Fig}(8)) \to S_*(LM \times LM, (LM \times LM) - \mathrm{Fig}(8)), \]
which is a chain equivalence by excision. By standard arguments, a chain homotopy inverse
\[ e \colon S_*(LM \times LM, (LM \times LM) -\mathrm{Fig}(8)) \to S_*(V_\epsilon, V_\epsilon - \mathrm{Fig}(8))\]
can be defined by using barycentric subdivision. By picking such an inverse, we guarantee that if $c$ is a chain consisting of pairs of piecewise-linear loops supported in a sequence of simplices, that is, supported in a product of the cells of \cref{def:piecewiseLinearCells}, then so is $e(c)$. We use the piecewise-linear connecting map $\varpi_\mathrm{pl}$ to define a joining map
\[ \mathrm{join} \colon V_\epsilon \to LM, \]
which also has the property of preserving piecewise linearity. The chain-level Chas--Sullivan product is then the map $\mu_{CS}$ of degree $-n$ given by the composition
\begin{align*}
    S_*(LM) \otimes S_*(LM) &\to S_*(LM \times LM) \to S_*(LM \times LM, (LM \times LM) -\mathrm{Fig}(8)) \to \\
    &\xrightarrow{e} S_*(V_\epsilon, V_\epsilon - \mathrm{Fig}(8)) \xrightarrow{ (-1)^n \theta_\mathrm{CS} \cap} S_*(V_\epsilon) \xrightarrow{S_*\mathrm{join}} S_*(LM),
\end{align*}
noting that we twisted by a sign $(-1)^n$ when capping with the Thom cochain.

As for the Goresky--Hingston coproduct, consider the space
\[ E \coloneqq \{((\gamma, T), t) \in LM \times \R \ |\ 0 \le t \le T\}\subset LM \times \R \]
and the map \[\mathrm{ev} \colon E \to M \times M\] defined by
\[ \mathrm{ev}(((\gamma, T), t))= (\gamma(0), \gamma(t)).\]
We then construct the pullbacks
\[\xymatrix{
    \mathrm{Fig}(8)' \ar@{^{(}->}[r] \ar[d] & W_\epsilon \ar[d] \ar@{^{(}->}[r] & E \ar[d]^{\mathrm{ev}} \\
    M \ar@{^{(}->}[r] & N_\epsilon \ar@{^{(}->}[r] & M \times M
}\]
and let $\theta_{\mathrm{GH}}=\mathrm{ev}^*\tau$. We have a degree $+1$ map
\[ m \colon S_*(LM) \to S_{*}(E) \]
that can be explicitly constructed through a simplicial subdivision of the prisms $\triangle^k\times \triangle^1$. This is not a chain map, but rather has boundary contained in the locus
\[ F \coloneqq \{((\gamma, T), t)\in E \ |\ t=0,T\} \subset E.\]
Using again the piecewise-linear connecting map $\varpi_\mathrm{pl}$ we define a cutting map
\[ \mathrm{cut} \colon W_\epsilon \to LM \times LM, \]
 which also preserves piecewise linearity. The Goresky--Hingston coproduct is then the map $\lambda_{GH}$ of degree $1-n$ given by the composition of chain maps
\begin{align*}
    S_*(LM,M)  &\xrightarrow{m} S_*(E,F) \to S_*(E,F \cup (E-\mathrm{Fig}(8)')) \xrightarrow{e'} S_*(W_\epsilon,  F \cup (W_\epsilon-\mathrm{Fig}(8)')) \to \\
    & \xrightarrow{(-1)^n\theta_\mathrm{GH} \cap} S_*(W_\epsilon,  F) \xrightarrow{S_*\mathrm{cut}} S_*(LM \times LM, (M \times LM) \cup (LM \times M)),
\end{align*}
where again $e'$ is an excision map defined using barycentric subdivision. Note that, by construction, the maps $\mu_\mathrm{CS}$ and $ \lambda_\mathrm{GH}$ preserve piecewise linearity.

\begin{remark}
    In \cite[Appendix~B]{HW}, there is a discussion about signs for the intersection product. We here take the perspective that the operations that those authors denote $(x,y) \mapsto x \bullet_\mathbf{Th} y$ and $(x,y) \mapsto x \bullet_\mathbf{P} y$ should be seen as morphisms from different source complexes. Namely, the `Thom product' does come from a degree $-n$ chain map
    \[ (S_*(M) \otimes S_*(M), \delta \otimes 1 + 1 \otimes \delta) \to (S_*(M), \delta), \]
    where the source has the usual differential on the tensor product, while the `Poincar\'e product' does not necessarily. Instead, it should be thought of as the map on homology induced by a degree $-n$ chain map
    \[ (S_*(M) \otimes S_*(M), \delta \otimes 1 + (-1)^x 1 \otimes \delta ) \to (S_*(M),\delta), \]
    where $x$ denotes the first input. It seems better for us to only deal with one type of binary operation, writing instead
    \[ m_{\mathbf{Th}}(x,y) \coloneqq x \bullet_\mathbf{Th} y \quad \text{and} \quad m_{\mathbf{P}}(x,y) \coloneqq  (-1)^{np} x \bullet_\mathbf{P} y \]
    and including the degree of the symbols $m_{\dots}$ whenever signs are created by the Koszul rule. In that notation, the relation between these products becomes simply
    \[ m_{\mathbf{Th}}(x,y) = (-1)^n m_{\mathbf{P}}(x,y). \]
    Therefore, by including the factor $(-1)^n$ in our geometric definitions of $\mu_\mathrm{CS}, \lambda_\mathrm{GH}$, we get the Poincar\'e versions of the Chas--Sullivan product and Goresky--Hingston coproduct, respectively denoted $\wedge$ and $\vee$ in \textit{op.cit.}, but seen as chain maps for the usual differential on the tensor product.
\end{remark}

We let $K$ be a $1$-fine simplicial complex obtained from a triangulation of $M$ and denote by
\[ \theta \colon C_*(K) \otimes C_*(K) \to R \]
the pairing of degree $-n$ on the given by mapping to singular chains on $M$ and evaluating against the Thom cochain $\tau$; this is a nondegenerate local pairing, in the sense of \cref{def:thetaNondeg}. Picking any coherent ordering defines a simplicial set $X$, and the pairing then gives rise to a local homotopy pairing $\alpha$ on $\mathcal{C}(\widetilde{X})$, canonically defined up to an exact term, by \cref{thm:alphaExistence,prop:alphaUnique}. This defines the associated algebraic loop product $\mu_\alpha$ and algebraic loop coproduct $\lambda_\alpha$.

\subsubsection{Comparison of products}
We now compare the algebraic loop product $\mu_{\alpha}$ on $\coCH_*(\mathcal{C}(\widetilde{X}))$ with the Chas--Sullivan product on $S_*(LM)$, using the quasi-isomorphism
\[ q \coloneqq S_*(Lf_X) \circ \ell_{\widetilde{X}} \colon \coCH_*(\mathcal{C}(\widetilde{X}))\to S_*(L|X|) \cong S_*(LM).\]

\begin{theorem}\label{thm:productComparison}
    The following square of complexes commutes up to chain homotopy
    \[\xymatrix{
        \coCH_*(\calC(\widetilde{X})) \otimes \coCH_*(\calC(\widetilde{X})) \ar[r]^-{\mu_\alpha} \ar[d]_{q \otimes q} &  \coCH_*(\calC(\widetilde{X})) \ar[d]^{q} \\
        S_*(LM) \otimes S_*(LM) \ar[r]^-{\mu_{\mathrm{CS}}} & S_*(LM).
    }\]
\end{theorem}
\begin{proof}
    We will start by assuming that our choice of Thom cochain $\tau$ vanishes on degenerate singular chains, that is, chains that factor through simplices of lower dimension; this can be always guaranteed by shifting by an exact term. We then construct the desired homotopy $h$ inductively on the total degree of the inputs, starting with the base case in the first nontrivial degree (when the sum of the degrees of the two inputs is $n$.) Let
    \[ x \otimes y = \sigma_0\{\sigma_1|\dots|\sigma_r \} \otimes \tau_0\{\tau_1|\dots|\tau_s \} \in \coCH_*(\calC(\widetilde{X})) \otimes \coCH_*(\calC(\widetilde{X}))  \]
    be a generator. There are two mutually exclusive possibilities:
    \begin{enumerate}
        \item $\dim(\sigma_0) + \dim(\tau_0) < n$, in which case the image of this input under $\mu_\alpha$ vanishes for degree reasons, and its image under $q\otimes q$ is degenerate so $\mu_\mathrm{CS}$ vanishes on it by assumption. In this case both sides are zero and we need no homotopies.
        \item $\dim(\sigma_0) + \dim(\tau_0) = n$, in which case the image of this input under both maps $q \circ \mu_\alpha$ and $\mu_\mathrm{CS} \circ (q \otimes q)$ is a linear combination of piecewise-linear loops with total multiplicity $\theta(\sigma_0,\tau_0)$, the former by the characterization of $\alpha$ in \cref{thm:alphaExistence}, together with the computation of the Koszul sign in \cref{def:algebraicLoopProduct}, and the latter by construction of the Chas--Sullivan product. So they both represent the same class in the zeroth homology group of the cell in $LM$ specified in \cref{def:piecewiseLinearCells} by the sequence of simplices
        \[ \sigma_0, \sigma_1, \dots, \sigma_r, \sigma_0,\tau_0, \tau_1,\dots,\tau_s,\tau_0 \]
        and therefore we can pick $h(x,y)$ to be given by any null-homology of their difference that is contained in that same cell.
    \end{enumerate}
    We can then proceed degree by degree inductively, since at each step the equation to be satisfied by $h(x,y)$ is of the form $d h(x,y) = z$ where $z$ is a closed element of strictly positive degree and with support contained inside one of the contractible cells we defined in $LM$, so we can pick any primitive of that chain contained inside the same cell.
\end{proof}

Taking homology, we immediately obtain the following.
\begin{corollary}\label{cor:loopproduct}
    The homology loop product induced by $\mu_\alpha$ is equivalent to the Chas--Sullivan product on $LM$ under the isomorphism $\mathcal{coHH}_*(\calC(\widetilde{X}))\cong H_*(LM)$.
\end{corollary}

\subsubsection{Comparison of coproducts}
We keep the same $(M, K,X, \theta,\alpha)$ as above, and let us assume that $X$ is $3$-fine. Throughout this subsection, for simplicity of notation, we write $\mathcal{C}=\mathcal{C}(\widetilde{X})$. We shall compare the algebraic coproduct $\lambda_{\alpha}$ on an appropriate quotient of $\coCH_*(\mathcal{C})$ with the Goresky--Hingston coproduct on $S_*(LM,M)$.

Recall that in our model for the free loop space in terms of the coHochschild complex we proved that we can replace constant loops by any small enough thickening, in this case up to $3$-local chains. We now describe a similar thickening of the constant loops inside singular chains in $LM$, which will be convenient for writing the comparison of the coproducts.
Let us denote by
\[ S_*(LM)_{3\mathrm{-loc}} = \sum_{Z} S_*(L|Z|) \subset S_*(LM), \]
where $Z$ ranges over all closed simplicial subcomplexes of $K$ with diameter $\le 3$. This subcomplex contains $q(\coCH_*(\calC)_{3\mathrm{-loc}})$ since $q$ is support-preserving.
\begin{proposition}
    The inclusion $q(\coCH_*(\calC)_{3\mathrm{-loc}}) \subset S_*(LM)_{3\mathrm{-loc}}$ is a quasi-isomorphism, and this latter is preserved by the chain-level Goresky--Hingston coproduct, in the sense that
    \[ \lambda_\mathrm{GH}\left(S_*(LM)_{3\mathrm{-loc}} \right) \subset S_*(LM)_{3\mathrm{-loc}} \otimes S_*(LM)_{3\mathrm{-loc}}. \]
\end{proposition}
\begin{proof}
    The first statement follows from the same argument as we used in the proof of \cref{prop:mLocalEquivalences}. As for the second statement, it follows directly from our definition of the Goresky--Hingston coproduct.
\end{proof}

Together with \cref{prop:mLocalEquivalences}, this implies that the Goresky--Hingston coproduct factors through the quotient by this larger complex, in other words there is a bottom map completing the square
\[\xymatrix{
    \dfrac{S_*(LM)}{S_*(M)} \ar[d] \ar[rr] & &  \dfrac{S_*(LM) \otimes S_*(LM)}{S_*(M)\otimes S_*(LM) + S_*(LM)\otimes S_*(M)} \ar[d] \\
    \dfrac{S_*(LM)}{S_*(LM)_{3\mathrm{-loc}}} \ar[rr] & & \dfrac{S_*(LM) \otimes S_*(LM)}{S_*(LM)_{3\mathrm{-loc}}\otimes S_*(LM) + S_*(LM)\otimes S_*(LM)_{3\mathrm{-loc}}}
}\]
whose vertical maps are quasi-isomorphisms. Therefore we can use this bottom map, which we also call $\lambda_\mathrm{GH}$, when writing the comparison.
\begin{theorem}
    The following square of complexes commutes up to chain homotopy
    \[\xymatrix{
        \dfrac{\coCH_*(\calC)}{\coCH_*(\calC)_{3\mathrm{-loc}}} \ar[d] \ar[rr]^-{\lambda_\alpha} & &  \dfrac{ \coCH_*(\calC) \otimes \coCH_*(\calC)}{(\coCH_*(\calC))_{3\mathrm{-loc}}\otimes \coCH_*(\calC) + \coCH_*(\calC)\otimes (\coCH_*(\calC))_{3\mathrm{-loc}}} \ar[d] \\
        \dfrac{S_*(LM)}{S_*(LM)_{3\mathrm{-loc}}} \ar[rr]^-{\lambda_\mathrm{GH}} & & \dfrac{S_*(LM) \otimes S_*(LM)}{S_*(LM)_{3\mathrm{-loc}}\otimes S_*(LM) + S_*(LM)\otimes S_*(LM)_{3\mathrm{-loc}}}.
    }\]
\end{theorem}
\begin{proof}
    This proof will be similar to the proof of \cref{thm:productComparison}, but with the added complication that it becomes necessary to keep track of the supports of different terms in the outputs of the two maps, in order to make the contractibility part of the argument work. For each generator $x=  \sigma_0\{\sigma_1|\dots|\sigma_r \}$ of $\coCH_*(\calC)$, we underline factors whose support touches $\sigma_0$. That is, we underline $\sigma_0$ and all the other $\sigma_i$ such that
   $ \sigma_0 \cap \sigma_i \neq \varnothing$
    obtaining a symbol of the form
    \[ \underline{\sigma_0\left\{\sigma_1|\dots|\sigma_i \right.}|\sigma_{i+1}|\dots|\sigma_j|\underline{\sigma_{j+1}|\dots|\sigma_k}|\sigma_{k+1}|\dots|\sigma_l|\underline{\left.\sigma_{l+1}|\dots|\sigma_r \right\}}, \]
    for example. We look at groups of contiguous underlines, remembering that the lines `wrap around'; for example, the above element has \emph{two} contiguous underlines. Note that there is always one such underline containing $\sigma_0$.
    
    We first note that if \emph{all} the $\sigma_i$ are underlined, then all of these simplices are supported at distance one from $\sigma_0$ so there is a path of length $\le 3$ between any two vertices in the support. Thus this generator is in $\coCH_*(\calC)_{3\mathrm{-loc}}$ and maps to zero in the quotient. Likewise, if there is exactly one contiguous group, the one containing $\sigma_0$, then the output of both maps is zero in the quotient, since at least one of the loops in each tensor is $3$-local.

    To complete the proof, we must specify the value of the homotopy on inputs that have two or more contiguous underlines. We will do so inductively while using the data of these underlines to strengthen the inductive hypothesis, in the following way. Let us suppose that, besides the group containing $\sigma_0$, there are $N$ contiguous underlines, with $N$ groups of underlined simplices
    \[ \underline{\sigma_{i_s}, \sigma_{i_s+1},\dots,\sigma_{j_s - 1}, \sigma_{j_s}} \]
    where $s=0,\dots,N$ and $i_s,j_s$ are the appropriate pair of indices delimiting the contiguous underline. In each such group labeled by some $s$, since all simplices are adjacent to $\sigma_0$, their union has diameter $\le 3$ and therefore by assumption there is a contractible simplicial subcomplex $Z_s$ containing them.

    We use these contractible subcomplexes to define cells of piecewise-linear loops by taking unions of cells of the form defined by \cref{def:piecewiseLinearCells}. For each $s$ let us denote 
    \[ S'_s = \bigcup_{\tau \in Z_s} S(\sigma_0,\sigma_1,\dots,\sigma_{i_s-1},\tau), \qquad S''_s = \bigcup_{\tau \in Z_s} S(\sigma_0,\tau,\sigma_{j_s + 1},\dots,\sigma_r). \]
    These are contractible subsets of $LM$, since each $Z_s$ is contractible. Since the underlined simplices are exactly the ones on which $\alpha(\sigma_0,-)$ can evaluate non-trivially, it follows immediately from the definition of the algebraic loop coproduct that we then have a decomposition
    \[ \lambda_\alpha(x) = \sum_{s=1}^N \lambda_{\alpha,s}(x), \]
    modulo terms in $(\coCH_*(\calC))_{3\mathrm{-loc}}\otimes \coCH_*(\calC) + \coCH_*(\calC)\otimes (\coCH_*(\calC))_{3\mathrm{-loc}}$, where we can decompose each term as a sum of simple tensors of generators
    \[ \lambda_{\alpha,Z_s}(x) = \sum_{a \in I} y'_{s,a} \otimes y''_{s,a} \] 
    with $\Supp(y'_{s,a}) \subset S'_s$ and $\Supp(y''_{s,a}) \subset S''_s$. Note that we can ignore the terms that come from the contiguous underline containing $\sigma_0$, since they are automatically in \[(\coCH_*(\calC))_{3\mathrm{-loc}}\otimes \coCH_*(\calC) + \coCH_*(\calC)\otimes (\coCH_*(\calC))_{3\mathrm{-loc}}.\]

    The analogous decomposition exists for the Goresky--Hingston coproduct. To see this, using the notation we used in \cref{sec:recap}, we look at the singular chain
    \[ y = e' \circ m \circ q(\sigma_0\{\sigma_1|\dots|\sigma_r \}) 
    \in S_*(W_\epsilon,  F \cup (W_\epsilon-\mathrm{Fig}(8)')) \]
    and pick any representative $\hat{y} \in S_*(W_\epsilon)$. We note that if $(\gamma,t)$ is in the support of $\hat{y}$ and $t \neq 0,1$, then we know that the loop $\gamma$ traverses $\sigma_0,\sigma_1,\dots,\sigma_r$ in sequence, with origin in $\sigma_0$, and that the point $\gamma(t)$ is in the support of one of the underlined simplices, since it is $\epsilon$-close to the origin and we chose $\epsilon$ to satisfy \cref{prop:epsilonSimplices}. Thus, due to the manner in which we defined $\lambda_\mathrm{GH}$, using the piecewise-linear connecting map, and since the excision map $e'$ was defined by barycentric subdivision, we have a decomposition
    \[
        \lambda_{\mathrm{GH}} \circ q(x) = \sum_{s=1}^N \lambda_{\mathrm{GH},s}(x),
    \]
    modulo terms in $S_*(LM)_{3\mathrm{-loc}}\otimes S_*(LM) + S_*(LM)\otimes S_*(LM)_{3\mathrm{-loc}}$, where again we can decompose
    \[ \lambda_{\mathrm{GH},s}(\sigma_0\{\sigma_1|\dots|\sigma_r \}) = \sum_{a \in I} z'_{s,a} \otimes z''_{s,a} \] with $\Supp(z'_{s,a}) \subset S'_s$ and $\Supp(z''_{s,a}) \subset S''_s$. 
    We are now ready to run our induction with a strengthened statement. We will prove that there is a homotopy $h$ between $(q \otimes q) \circ \lambda_\alpha$ and $\lambda_{\mathrm{GH}} \circ q$ satisfying the following locality condition. For any generator $x$ as above, we have a similar decomposition for the homotopy \[h(x) = \sum_{s=1}^N h_s(x)\] with $h_s(x) = \sum_c h_{s,c}'(x) \otimes h_{s,c}''(x)$ where $\Supp(h_{s,c}'(x)) \subset S'_s, \quad \Supp(h_{s,c}''(x)) \subset S''_s$, whose terms satisfy
    \[ d(h_s(x)) + \sum_{s=1}^N h_s(dx)|_{S'_s \times S''_s} = \lambda_{\alpha,s}(x) - \lambda_{\mathrm{GH},s}. \]
    The notation $( -)|_{S'_s \times S''_s}$ indicates all the terms given by a tensor of factors with support in $S'_s,S''_s$, respectively.

    To say this in words, by construction of the cells $S_s$, each of the cells on which some term in $h_s(dx)$ has support sits inside of one of the cells $S_0,\dots,S_N$; there may be multiple of them for each $s$. The statement above means that the decomposition $h = \sum h_s$ separates $h$ into a homotopy in each one of those cells.

    Now that we formulated this strengthened induction hypothesis, we can run our argument, staring with an input with $x$ as above with degree $(n-1)$. There are two possibilities:
    \begin{enumerate}
        \item there exists exactly one $i$ among $1,\dots,r$ such that $\dim(\sigma_0) + \dim(\sigma_i) = n$, or
        \item $\dim(\sigma_0) = n-1$ and all the $\sigma_1,\dots,\sigma_r$ are 1-simplices.
    \end{enumerate}
    In case (1), both $\lambda_{\alpha}$ and $\lambda_{\mathrm{GH}}$ are nontrivial only when $\sigma_i \cap \sigma_0 \neq \varnothing$, in which case they both have output supported $S'_s \times S''_s$ for a single such $s$, the one corresponding to the contiguous underline containing $\sigma_i$. Moreover, calculating the sign in the algebraic coproduct we find that both outputs represent the same class
    \[ \widetilde{\theta}(\sigma_0,\sigma_i) \times [\text{loop in } S'_s] \times [\text{loop in } S''_s]  \]
    so by contractibility of the cells $S'_s$ and $S''_s$ we can find a homotopy also contained in their product. The argument for case (2) is similar, except that there may be multiple indices $s$ on whose corresponding cells the expressions for $\lambda_{\alpha}$ and $\lambda_{\mathrm{GH}}$, and consequently the homotopy, have support. The induction can then be run just as we have done in the proof of \cref{thm:productComparison}, using contractibility of the cells $S'_s,S''_s$ at each stage.
\end{proof}

Passing to homology we immediately obtain the following.
\begin{corollary} \label{cor:loopcoproduct}
    Under the isomorphism $H_*\left(\coCH_*(\calC)/(\coCH_*(\calC))_{3\mathrm{-loc}}\right) \cong H_*(LM,M)$, the algebraic loop coproduct corresponds to the Goresky--Hingston coproduct. \\
\end{corollary}

\printbibliography

\end{document}